\documentclass[a4paper,11pt]{amsart}

\usepackage[utf8]{inputenc}		% font and encoding
\usepackage[T1]{fontenc}
\usepackage[english]{babel}

\usepackage{tikz,tikz-3dplot, forloop}

\usepackage{amsfonts}			% ams-packages
\usepackage{amsmath}
\usepackage{amssymb}
\usepackage{amsthm}

\usepackage{thmtools}
\usepackage{thm-restate}

\usepackage{graphicx}			% images and drawing
\usepackage{tikz}
	\usetikzlibrary{cd}

\usepackage{hyperref}			% hyperlinks
	\hypersetup{colorlinks=false, urlcolor=black, linkcolor=black}
\usepackage{cleveref}

\usepackage{enumitem}			% arbitrary list items

\usepackage{color}				% colored text for signaling changes
\usepackage{float}

\usepackage{mathrsfs}

\newcommand{\Z}{\mathbb{Z}}						% Integers
\newcommand{\R}{\mathbb{R}}						% Reals
\newcommand{\C}{\mathbb{C}}						% Complex numbers

\renewcommand{\S}{\mathbb{S}}					% Sphere
\newcommand{\B}{\mathbb{B}}
\newcommand{\eps}{\varepsilon}					% Epsilon shortcut
\newcommand{\dd}								% Differential d
	{\mathop{}\!\mathrm{d}}						
\newcommand{\ddn}[1]							% Powers of a differential d
	{\mathop{}\!\mathrm{d^{#1}}}

\newcommand{\abs}[1]							% Absolute value
	{\left| #1 \right|}
\newcommand{\smallabs}[1]						% Small absolute value bars which won't scale to the argument.
	{\lvert #1 \rvert}	
\newcommand{\norm}[1]							% Norm 
	{\left\lVert #1 \right\rVert}	
\newcommand{\smallnorm}[1]						% Small norm bars which won't scale to the argument.
	{\lVert #1 \rVert}						
\newcommand{\ip}[2]								% Inner product
	{\left< #1 , #2 \right>}

\DeclareMathOperator{\intr}{int}				% Interior
\DeclareMathOperator{\vol}{vol}					% Volume (and the volume form)
\DeclareMathOperator{\spt}{spt}					% Support
\DeclareMathOperator*{\esssup}{ess\,sup}		% Essential supremum

\DeclareMathOperator{\capac}{Cap}			
\DeclareMathOperator{\sgn}{sgn}

\DeclareMathOperator{\pseudosup}{\widetilde{\sup}}

\newcommand{\hodge}{\mathtt{\star}\hspace{1pt}}

\newcommand{\loc}{\mathrm{loc}}

\newcommand{\cA}{\mathcal{A}}

\newcommand{\cC}{\mathcal{C}}

\newcommand{\cH}{\mathcal{H}}

\newcommand{\cS}{\mathcal{S}}

\newcommand{\cX}{\mathcal{X}}

\renewcommand{\phi}{\varphi}

\def\Xint#1{\mathchoice
	{\XXint\displaystyle\textstyle{#1}}%
	{\XXint\textstyle\scriptstyle{#1}}%
	{\XXint\scriptstyle\scriptscriptstyle{#1}}%
	{\XXint\scriptscriptstyle\scriptscriptstyle{#1}}%
	\!\int}
\def\XXint#1#2#3{{\setbox0=\hbox{$#1{#2#3}{\int}$}
		\vcenter{\hbox{$#2#3$}}\kern-.5\wd0}}

\def\dashint{\Xint-}

\newtheorem{thm}{Theorem}[section]{\bf}{\it}
\newtheorem{lemma}[thm]{Lemma}
\newtheorem{prop}[thm]{Proposition}
\newtheorem{cor}[thm]{Corollary}

{\bf}{\it}

{\bf}{\it}

{\bf}{\it}

{\bf}{\it}

\theoremstyle{definition}

\newtheorem{ex}[thm]{Example}

\theoremstyle{remark}
\newtheorem{rem}[thm]{Remark}

\numberwithin{equation}{section}

\begin{document}

\title{Quasiregular values and Rickman's Picard theorem}

\author[I. Kangasniemi]{Ilmari Kangasniemi}
\address{Department of Mathematical Sciences, University of Cincinnati, P.O.\ Box 210025, Cincinnati, OH 45221, USA.}
\email{kangaski@ucmail.uc.edu}

\author[J. Onninen]{Jani Onninen}
\address{Department of Mathematics, Syracuse University, Syracuse,
NY 13244, USA and  Department of Mathematics and Statistics, P.O.Box 35 (MaD) FI-40014 University of Jyv\"askyl\"a, Finland
}
\email{jkonnine@syr.edu}

\date{\today}
\thanks{I. Kangasniemi was supported by the National Science Foundation grant DMS-2247469. J. Onninen was supported by the National Science Foundation grant DMS-2154943. }

\begin{abstract}
We prove a far-reaching generalization of Rickman's Picard theorem for a surprisingly large class of mappings, based on the recently developed theory of quasiregular values. Our results are new even in the planar case.
\end{abstract}

\maketitle

\section{Introduction}
Geometric Function Theory (GFT) is largely concerned with generalizations of the theory of holomorphic functions of one complex variable. A widely studied example is the theory of quasiregular maps, which provides such a generalization for spaces of real dimension $n \geq 2$. We recall that, given a domain $\Omega \subset \R^n$ and a constant $K \geq 1$, a \emph{$K$-quasiregular map} $f \colon \Omega \to \R^n$ is a continuous map in the Sobolev space $W^{1,n}_\loc(\Omega, \R^n)$ which satisfies the distortion inequality
\begin{equation}\label{eq:QR_def}
	\abs{Df(x)}^n \leq K J_f(x)
\end{equation}
for almost every (a.e.)\ $x \in \Omega$. Here, $\abs{Df(x)}$ is the operator norm of the weak derivative of $f$ at $x$, and $J_f$ denotes the Jacobian determinant of $f$.

A significant achievement in the theory of higher-dimensional quasiregular maps is the extension of the classical Picard theorem to $n$ real dimensions. This highly non-trivial result is due to Rickman~\cite{Rickman_Picard}.

\begin{thm}[Rickman's Picard Theorem]\label{thm:Rickman_Picard}
	For every $K \geq 1$ and  $n \geq 2$, there exists a positive integer $q = q(n, K) \in \Z_{> 0}$ such that if $f \colon \R^n \to \R^n$ is $K$-quasiregular and $\R^n \setminus f(\R^n)$ contains $q$ different points, then $f$ is constant.
\end{thm}

Rickman's theorem leaves an impression that the global distortion control of quasiregular mappings is necessary for the bound on the number of omitted points. However, in this article, we show that the distortion bound only needs to hold in an asymptotic sense when $f$ is near the omitted points, and can in fact be replaced with an appropriate Sobolev norm estimate elsewhere. Our result is formulated using the recent theory of quasiregular values \cite{Kangasniemi-Onninen_1ptReshetnyak}. In particular, supposing that $y_0 \in \R^n$ and that $\Omega$ is a domain in $\R^n$ with $n \ge 2$, a map $f \colon \Omega \to \R^n$ in the Sobolev space $W^{1,n}_\loc(\Omega, \R^n)$ has a \emph{$(K, \Sigma)$-quasiregular value at $y_0$} if it satisfies the inequality 
\begin{equation}\label{eq:QRvalue_def}
	\abs{Df(x)}^n \leq K J_f(x) + \abs{f(x)-y_0}^n \Sigma(x)
\end{equation}
for a.e.\ $x \in \Omega$, where $K \geq 1$ is a constant as in \eqref{eq:QR_def} and $\Sigma$ is a nonnegative function in $L^{1+\eps}_\loc(\Omega)$ for some $\eps >0$. 

Notably, \eqref{eq:QRvalue_def} only provides control on the distortion of a function $f$ as $f(x)$ equals or asymptotically approaches $y_0$. Away from $y_0$, these functions can instead behave similarly to an arbitrary map in $W^{n+n\eps}_\loc(\Omega, \R^n)$. For instance, a non-constant map $f$ satisfying \eqref{eq:QRvalue_def} may for instance have a Jacobian that changes sign, an entirely 1-dimensional image, or a bounded image even when $f$ is defined in all of $\R^n$.  In addition, a map $f$ satisfying~\eqref{eq:QRvalue_def} needs not be locally quasiregular even in any neighborhood of a point $x_0 \in f^{-1} \{y_0\}$; in fact, it is possible that every neighborhood of such a point meets a region where $J_f <0$.

In spite of these vast differences, Rickman's Picard theorem still generalizes to the theory of quasiregular values in the following form.

\begin{restatable}{thm}{mainthm}\label{thm:Picard_for_QR_values}
	Let $K \geq 1$ and $\Sigma \in L^{1+\eps}(\R^n) \cap L^{1-\eps}(\R^n)$ for some $\eps > 0$. Then there exists a positive integer $q = q(n, K) \in \Z_{> 0}$ such that no  continuous map $f \in W^{1,n}_\loc(\R^n, \R^n)$ has a $(K, \Sigma)$-quasiregular value at $q$ distinct points $y_1, \dots, y_q \in \partial f(\R^n)$.
\end{restatable}

While the standard Rickman's Picard theorem concerns omitted points $y_i \notin f(\R^n)$, Theorem \ref{thm:Picard_for_QR_values} reveals that at this level of generality, Rickman's Picard Theorem is in fact a result on the number of points $y_i$ in the boundary $\partial f(\R^n)$. Indeed, a version of Theorem \ref{thm:Picard_for_QR_values} that instead assumes $y_1, \dots, y_q \notin f(\R^n)$ is immediately shown to be false by every single smooth compactly supported map $f \in C^\infty_0(\R^n, \R^n)$.  Regardless of this difference in statements, the standard Rickman's Picard theorem follows almost immediately from the  case $\Sigma \equiv 0$ of Theorem \ref{thm:Picard_for_QR_values}; see Remark \ref{rem:deriving_Rickman's_Picard}.

The integrability assumptions on $\Sigma$ in Theorem \ref{thm:Picard_for_QR_values} are sharp on the $L^p$-scale. Indeed, we show in Section \ref{sect:counterexamples} that  neither $\Sigma \in L^{1+\eps}_\loc(\R^n) \cap L^1(\R^n) \cap L^{1-\eps}(\R^n)$ nor $\Sigma \in L^{1+\eps}(\R^n) \cap L^{1}(\R^n)$ is sufficient for the result. The constructed maps even satisfy~\eqref{eq:QRvalue_def} with $K = 1$. We however expect a logarithmic Orlicz-type sharpening of the integrability assumptions to be possible, though we elect not to pursue log-scale results in this work unless explicitly required by an argument.

\subsection{Background on quasiregular maps and the Picard theorem}

The classical Picard theorem states that if $f \colon \C \to \C$ is an entire holomorphic function, then either $f$ is constant or $\C \setminus f(\C)$ contains at most one point. The Picard theorem is among the most striking and universally known results in complex analysis, with numerous different proofs discovered over the years: see e.g.~\cite{Ahlfors-Conformal-invariants, Borel-Sur, David-Picard, Fuchs-Topics, Hayman-Meromorphic, Lewis_Picard-with-Harnack, Segal-Nine, Zhang-Curves}.

The theory of quasiregular maps originates from the planar setting, with roots in the work of  Gr\"otzsch \cite{Grotzsch_QC} and Ahlfors \cite{Ahlfors_QC}. More specifically, when $n = 2$, the distortion inequality \eqref{eq:QR_def} can be rewritten as a linear Beltrami-type equation
\begin{equation}\label{eq:QR_def_Beltrami}
	f_{\overline{z}} = \mu f_z,
\end{equation}
where $f_z, f_{\overline{z}}$ are the (weak) Wirtinger derivatives of $f$ and $\mu \in L^\infty(\Omega, \C)$ satisfies $\norm{\mu}_{L^\infty} \leq k < 1$ with $k = (K-1)/(K+1)$. If $K = 1$, then \eqref{eq:QR_def_Beltrami} reduces to the Cauchy-Riemann system $f_{\overline{z}} = 0$; indeed, a planar map is $1$-quasiregular exactly if it is holomorphic. Moreover, homeomorphic solutions of \eqref{eq:QR_def} or \eqref{eq:QR_def_Beltrami} are called \emph{$K$-quasiconformal}, and we also have that a map is $1$-quasiconformal exactly if it is a conformal transformation. 

In addition to this link to holomorphic maps, planar quasiregular maps satisfy the Sto\"ilow factorization theorem, which states that a  quasiregular map $f \colon \Omega \to \C$ is of the form $f = h \circ g$ where $g \colon \Omega \to \Omega$ is quasiconformal and $h \colon \Omega \to \C$ is holomorphic, see e.g.\ \cite[Chapter 5.5]{Astala-Iwaniec-Martin_Book}. The Sto\"ilow factorization theorem immediately generalizes the topological properties of holomorphic maps to planar quasiregular maps, such as the open mapping theorem, the Liouville theorem, and even the Picard theorem.

The higher-dimensional version of the theory began with the study of $n$-dimensional quasiconformal mappings by e.g.\ \u{S}abat \cite{Sabat_QC}, V\"ais\"al\"a \cite{Vaisala_QCInSpace}, Gehring \cite{Gehring_TAMS}, and Zori\v{c} \cite{Zorich_QR}. Afterwards, the theory of $n$-dimensional quasi\-regular mappings was originated by Reshetnyak \cite{Reshetnyak_continuity,Reshetnyak_QROrigin, Reshetnyak_Theorem2, Reshetnyak_Liouville}, with significant early contributions by Martio, Rickman, and V\"ais\"al\"a \cite{Martio-Rickman-Vaisala_AASF-QR1, Martio-Rickman-Vaisala_AASF-QR2, Martio-Rickman-Vaisala_AASF-QR3}. The theory is by now a central topic in modern analysis, with important connections to partial differential equations, complex dynamics, differential geometry and the calculus of variations; see the textbooks of V\"ais\"al\"a \cite{Vaisala_book}, Rickman \cite{Rickman_book}, Reshetnyak \cite{Reshetnyak-book}, and Iwaniec and Martin \cite{Iwaniec-Martin_book}.

Unlike in the planar case, one cannot reduce the topological properties of higher dimensional quasiregular maps to a better understood class of mappings. Indeed, the best known Sto\"ilow-type theorem in higher dimensions \cite{Martin-Peltonen_UQRStoilow} still has a relatively irregular non-injective component. Nevertheless, many topological properties of holomorphic maps have non-trivial extensions to spatial quasiregular mappings as well. For instance, the open mapping theorem generalizes to Reshetnyak's theorem \cite{Reshetnyak_QROrigin, Reshetnyak_Theorem2}, which states that if $f \colon \Omega \to \R^n$ is a non-constant quasiregular map, then $f$ is an open, discrete map with positive local index $i(x, f)$ at every $x \in \Omega$. 

Rickman's Picard theorem, stated in Theorem~\ref{thm:Rickman_Picard}, is perhaps the most clear demonstration of the similarities between the theory of higher dimensional quasiregular maps and single-variable complex analysis. Consequently, it has become one of the most widely studied results in quasiregular analysis. For instance, a version of Rickman's Picard Theorem has been shown for quasiregular maps $f \colon \R^n \to M$ into an oriented Riemannian $n$-manifold $M$ by Holopainen and Rickman \cite{Holopainen-Rickman_Picard, Holopainen-Rickman_Picard-noncompact}. A version of the theorem has also been shown by Rajala \cite{Rajala_MFD-Picard} in the case where $f$ is a mapping of finite distortion, i.e.\ a mapping satisfying \eqref{eq:QR_def} with a non-constant $K$.

When $n = 2$, the Sto\"ilow factorization approach yields that the constant $q(2, K)$ in Rickman's Picard theorem is equal to $2$, and is thus in fact independent on $K$. It was conjectured for some time that one could also have $q(n, K) = 2$ for all $n \geq 2$ and $K \geq 1$. However, counterexamples by Rickman \cite{Rickman_PicardCounterex} in the case $n = 3$, and by Drasin and Pankka \cite{Drasin-Pankka_Acta} in the case $n \geq 4$, show that for a fixed $n > 2$ one has $q(n, K) \to \infty$ as $K \to \infty$. 

\subsection{The theory of quasiregular values}

Various generalizations of \eqref{eq:QR_def} and \eqref{eq:QR_def_Beltrami} occur in the study of complex analysis. For instance, the condition 
\begin{equation}\label{eq:simontype}
\abs{Df}^2 \leq K J_f + C\, , 
\end{equation}
where $K\ge 1$ and $C\ge 0$ are constants, arises naturally in the theory of elliptic PDEs~\cite[Chapter 12]{Gilbarg-Trudinger_Book}. The H\"older regularity of planar domain solutions of \eqref{eq:simontype} has been shown by Nirenberg~\cite{Nirenberg_Holder}, Finn and Serrin~\cite{Finn_Serrin}, and Hartman~\cite{Hartman}. Similar ideas also play a key role in the work of Simon \cite{Simon_QRGauss}, where he obtains H\"older estimates for solutions of \eqref{eq:simontype} between surfaces, and applies them to the study of equations of mean curvature type.

The theory of quasiregular values stems from another similar generalization of \eqref{eq:QR_def} and \eqref{eq:QR_def_Beltrami}, namely
\begin{equation}\label{eq:QRval_zero_Beltrami}
	f_{\overline{z}} = \mu f_z + A f,
\end{equation}
where $\norm{\mu}_{L^\infty} < 1$ and $A \in L_{\loc}^{2+\eps} (\Omega, \C)$ for some $\eps>0$. In particular, \eqref{eq:QRval_zero_Beltrami} corresponds to the case $n=2, y_0 = 0$ of definition \eqref{eq:QRvalue_def} of quasiregular values. Much of the initial theory on solutions of \eqref{eq:QRval_zero_Beltrami} was developed by Vekua \cite{Vekua_book}. One of the standout applications for \eqref{eq:QRval_zero_Beltrami} arose when Astala and P\"aiv\"arinta used it in their solution to the planar Calder\'on problem \cite{Astala-Paivarinta}. The solutions of~\eqref{eq:QRval_zero_Beltrami} play a key part of various other uniqueness theorems as well; we refer to the book of Astala, Iwaniec and Martin~\cite{Astala-Iwaniec-Martin_Book} for details.

Astala and P\"aiv\"arinta relied on two results for entire solutions of \eqref{eq:QRval_zero_Beltrami}, which were essentially modeled on the Liouville theorem and the argument principle; see \cite[Proposition 3.3]{Astala-Paivarinta} and \cite[Sect.\ 8.5 and 18.5]{Astala-Iwaniec-Martin_Book}. The original key idea behind the planar results is that any solution $f$ of \eqref{eq:QRval_zero_Beltrami} is of the form $f = g e^{\theta}$, where $g$ is quasiregular and $\theta \colon \Omega \to \C$ is a solution of $\theta_{\overline{z}} = \mu \theta_z + A$. Since the existence theory of Beltrami equations and the aforementioned decomposition $f = g e^{\theta}$ lack higher-dimensional counterparts, this planar approach fails to generalize to the $n$-dimensional setting. Nevertheless, we have recently in \cite{Kangasniemi-Onninen_Heterogeneous, Kangasniemi-Onninen_1ptReshetnyak} managed to obtain higher-dimensional counterparts to the planar results used by Astala and P\"aiv\"arinta. The Liouville-type theorem stated in \cite[Theorem 1.3]{Kangasniemi-Onninen_Heterogeneous} in particular answers the Astala-Iwaniec-Martin uniqueness question from \cite[Sect.\ 8.5]{Astala-Iwaniec-Martin_Book}, though it bears mention that we later discovered the original proof of this specific result to have a small but fatal flaw in the part \cite[Lemma 7.2]{Kangasniemi-Onninen_Heterogeneous} involving integrability below the natural exponent, and have submitted a corrigendum \cite{Kangasniemi-Onninen_Heterogeneous_Corrigendum} which recovers the original theorem though a non-trivial fix.

The higher-dimensional results opened up an entirely new direction of study in GFT, as they led us to introduce the class of maps with quasiregular values in \cite{Kangasniemi-Onninen_1ptReshetnyak}. The term ``quasiregular value'' is partially motivated by the single-value versions of various foundational results of quasiregular maps that follow from \eqref{eq:QRvalue_def}. The other main motivation for the term is the fact that $K$-quasiregularity of a map $f \in W^{1,n}_\loc(\Omega, \R^n)$ can be fully characterized by $f$ having a $(K, \Sigma_{y})$-quasiregular value with $\Sigma_y \in L^{1+\eps}_\loc(\Omega)$ at every $y \in \R^n$; see \cite[Theorem 1.3]{Kangasniemi-Onninen_1ptReshetnyak}. 

The following theorem lists the two most notable current results of quasi\-regular values, which are the single-value versions of the Liouville theorem and Reshetnyak's theorem. They were shown in \cite{Kangasniemi-Onninen_Heterogeneous} and \cite{Kangasniemi-Onninen_1ptReshetnyak}, respectively, and are key components behind the higher-dimensional versions of the planar results for solutions of \eqref{eq:QRval_zero_Beltrami}. The addition of the Picard-type Theorem \ref{thm:Picard_for_QR_values} to this growing list of results furthers the evidence that quasiregular values have a rich theory comparable to that of quasiregular mappings.

\begin{thm}[{\cite[Theorem 1.2]{Kangasniemi-Onninen_Heterogeneous} and \cite[Theorem 1.2]{Kangasniemi-Onninen_1ptReshetnyak}}]\label{thm:single_value_results}
	Let $\Omega \subset \R^n$ be a domain, let $\eps > 0$, and let $f \in W^{1,n}_\loc(\Omega, \R^n)$ be a continuous map with a quasiregular value at $y_0 \in \R^n$, for given choices of $K \geq 1$ and $\Sigma \colon \Omega \to [0, \infty)$. Then the following results hold.
	\begin{enumerate}[label=(\roman*)]
		\item \label{item:Liouville} (Liouville theorem) If $\Omega = \R^n$, $\Sigma \in L^{1+\eps}_\loc(\R^n) \cap L^1(\R^n)$, and $f$ is bounded, then either $f \equiv y_0$ or $y_0 \notin f(\R^n)$.
		\item \label{item:Reshetnyak} (Reshetnyak's theorem) If $\Sigma \in L^{1+\eps}_\loc(\Omega)$ and if $f$ is not the constant function $f \equiv y_0$, then $f^{-1}\{y_0\}$ is discrete, the local index $i(x, f)$ is positive at every $x \in f^{-1}\{y_0\}$, and $f$ maps every neighborhood $U \subset \Omega$ of a point of $f^{-1}\{y_0\}$ to a neighborhood $f(U)$ of $y_0$.
	\end{enumerate}
\end{thm}

We note that by \cite[Theorem~1.1]{Kangasniemi-Onninen_Heterogeneous}, solutions $f \in W^{1,n}_\loc(\Omega, \R^n)$ of \eqref{eq:QRvalue_def} always have a continuous representative if $\Sigma \in L^{1+\eps}_\loc(\Omega)$ for some $\eps > 0$; see also \cite{Dolezalova-Kangasniemi-Onninen_MGFD-cont} which explores how much these assumptions can be relaxed for continuity to remain true. Hence, the assumption of continuity in our results only amounts to making sure that our chosen representative of the Sobolev map is the continuous one.

\subsection{Other versions of Theorem \ref{thm:Picard_for_QR_values}}

Besides the standard formulation for quasiregular mappings $f \colon \R^n \to \R^n$, Rickman's Picard theorem is often also equivalently formulated for quasiregular mappings $f \colon \R^n \to \S^n$. In our setting, we similarly obtain a version of Theorem \ref{thm:Picard_for_QR_values} for mappings $f \colon \R^n \to \S^n$ with little extra effort, though it requires formulating a spherical version of \eqref{eq:QRvalue_def}. Given $\Omega \subset \R^n$, $K \geq 0$, $y_0 \in \S^n$, and $\Sigma \in L^{1+\eps}_\loc(\Omega, [0, \infty))$ with $\eps > 0$, we say that a continuous mapping $h \in W^{1, n}(\Omega, \S^n)$ has a \emph{$(K, \Sigma)$-quasiregular value with respect to the spherical metric at $w_0 \in \S^n$} if $f$ satisfies
\begin{equation}\label{eq:spherical_QR_value}
	\abs{Dh(x)}^n \leq K J_h(x) + \sigma^n(h(x), w_0) \Sigma(x)
\end{equation}
at a.e.\ $x \in \Omega$, where $\sigma(\cdot, \cdot)$ denotes the spherical distance on $\S^n$, and $\abs{Dh(x)}$ and $J_h(x)$ are defined using the standard Riemannian metric and orientation on $\S^n$. With this definition, the resulting version of Theorem \ref{thm:Picard_for_QR_values} is as follows.

\begin{restatable}{thm}{mainthmsphere}\label{thm:Picard_for_QR_values_spherical}
	Let $K \geq 1$ and $\Sigma \in L^{1+\eps}(\R^n) \cap L^{1-\eps}(\R^n)$ for some $\eps > 0$. Then there exists $q = q(n, K) \in \Z_{> 0}$ such that no continuous map $h \in W^{1,n}_\loc(\R^n, \S^n)$ has a $(K, \Sigma)$-quasiregular value with respect to the spherical metric at $q$ distinct points $w_1, \dots, w_{q} \in \partial h(\R^n)$.
\end{restatable}

We remark that if we identify $\S^n$ with $\R^n \cup \{\infty\}$ via the stereographic projection, then a map $f \colon \R^n \to \R^n$ has a quasiregular value with respect to the Euclidean metric at $y_0 \in \R^n$ if and only if $f$ has a quasiregular value with respect to the spherical metric at both $y_0$ and $\infty$. Hence, \eqref{eq:spherical_QR_value} is in some sense a weaker assumption than \eqref{eq:QRvalue_def}. The comparison between these two definitions is discussed in greater detail in Section \ref{sect:qrvalues_into_Sn}.

While the assumption $\Sigma \in L^{1+\eps}(\R^n) \cap L^{1-\eps}(\R^n)$ in Theorems \ref{thm:Picard_for_QR_values} and \ref{thm:Picard_for_QR_values_spherical} is sharp, the proof we use does yield us some additional information even under a weaker assumption of $\Sigma \in L^{1}(\R^n) \cap L^{1+\eps}_\loc(\R^n)$. This result is more elegantly stated using spherical quasiregular values.

\begin{restatable}{thm}{mainthmBCpart}\label{thm:Bonk-PoggiCorradini-part}
	Let $K \geq 1$ and $\Sigma \in L^{1}(\R^n) \cap L^{1+\eps}_\loc(\R^n)$ for some $\eps > 0$. Then there exists $q = q(n, K) \in \Z_{> 0}$ as follows: if a continuous map $h \in W^{1,n}_\loc(\R^n, \S^n)$ has a $(K, \Sigma)$-quasiregular value with respect to the spherical metric at $q$ distinct points $w_1, \dots, w_{q} \in \partial h(\R^n)$, then $h \in W^{1,n}(\R^n, \S^n)$.
\end{restatable}

\subsection{The planar case} In the case $n = 2$, similarly to the standard Picard theorem, our main results end up having $q(2, K) = 2$ for maps $f \colon \C \to \C$, and $q(2, K) = 3$ for maps $h \colon \C \to \S^2$. Even this planar version of Theorem \ref{thm:Picard_for_QR_values} and Theorem \ref{thm:Picard_for_QR_values_spherical} is new.

\begin{restatable}{thm}{mainthmplanar}\label{thm:q_is_2_when_n_is_2}
	Let $K \geq 1$ and $\Sigma \in L^{1+\eps}(\C) \cap L^{1-\eps}(\C)$ for some $\eps > 0$. Then no continuous map $f \in W^{1,2}_\loc(\C, \C)$ has a $(K, \Sigma)$-quasiregular value at two distinct points $z_1, z_2 \in \partial f(\C)$. Similarly, no continuous map $h \in W^{1,2}_\loc(\C, \S^2)$ has a $(K, \Sigma)$-quasiregular value with respect to the spherical metric at three distinct points $w_1, w_2, w_3 \in \partial h(\C)$. 
\end{restatable}

We prove Theorem \ref{thm:q_is_2_when_n_is_2} by reducing it to Theorem \ref{thm:Picard_for_QR_values}. The version of the argument for quasiregular maps is incredibly simple: If $f \colon \C \to \C$ is a $K$-quasiregular map omitting two distinct points $z_1, z_2 \in \C$, then the lift $\gamma \colon \C \to \C$ of $f$ in the exponential map $z \mapsto z_1 + e^{z}$ is a $K$-quasiregular map that omits the infinitely many values of $\log(z_2 - z_1)$, which is impossible by Rickman's Picard Theorem. Attempting the same idea for maps with quasiregular values using Theorem \ref{thm:Picard_for_QR_values} is less straightforward, but we are ultimately able to construct a proof around this fundamental idea through use of the decomposition $f = g e^{\theta}$ and existing results on quasiregular values; see Section~\ref{sect:planar} for details.

\subsection{Main ideas of the proof}

While the classical Picard theorem has numerous proofs, only a few of them have been successfully generalized to a proof of the $n$-dimensional Rickman's Picard Theorem. The original proof by Rickman \cite{Rickman_Picard} uses path lifting and conformal modulus techniques in order to estimate spherical averages of the multiplicity function of $f$. Later, work by Eremenko and Lewis \cite{Eremenko-Lewis_Rickman-Picard, Lewis_Picard-with-Harnack} resulted in an alternate proof using Harnack inequalities of $\cA$-harmonic maps. Both of these approaches run into significant obstacles in our setting, as solutions of \eqref{eq:QRvalue_def} currently lack counterparts to e.g.\ conformal modulus estimates and the natural conformal structure $G_f(x) = J_f^{2/n}(x) [D^Tf(x) Df(x)]^{-1}$ of $f$.

Recently, however, a third method of proof has been discovered by Bonk and Poggi-Corradini \cite{Bonk-PoggiCorradini_Rickman-Picard}, which is closer to being applicable in our situation. Motivated by the Ahlfors-Shimizu -variant of value distribution theory, they study the pull-back under $f$ of a subharmonic logarithmic singularity function $v \colon \S^n\setminus\{x_0\} \to [0, \infty)$ such that the spherical $n$-Laplacian of $v$ is identically $1$. They are then able to leverage the preservation of the spherical measure under isometric rotations of $\S^n$ to obtain growth rate estimates for the measure $A_f = f^* \vol_{\S^n}$, from which the result follows via ideas reminiscent of the ones used in Rickman's original argument. 

We prove Theorem \ref{thm:Picard_for_QR_values} by adopting the structure of the proof of Bonk and Poggi-Corradini, but with key developments to the proof in multiple places where its current form is insufficient for us. Notably, in order to avoid use of the conformal structure $G_f$, we completely eliminate the use of $\cA$-subharmonic theory in our proof, and we instead obtain the required growth estimates by directly using \eqref{eq:QRvalue_def} and the properties of the function $v$. Issues caused by the extra term in \eqref{eq:QRvalue_def} and the fact that $A_f$ is a signed measure are eliminated by the global $L^1$-integrability of $\Sigma$.

The greatest challenges in our setting are tied to replacing the use of \cite[Lemma 4.4]{Bonk-PoggiCorradini_Rickman-Picard}, which yields that if $f \colon \R^n \to \R^n$ is a non-constant entire quasiregular map and $r > 0$, then every component of the set $\{\abs{f} > r\}$ is unbounded. In our case, this is not true; instead, we essentially obtain control on the total $A_f$-measure of any bounded components of $\{\abs{f} > r\}$. One of our primary tools in addressing this problem is to introduce a ``pseudosupremum'' based on unbounded components of pre-images. Indeed, when the growth estimates for $A_f$ are formulated in terms of this pseudosupremum, they can be combined in a similar manner as in the case of quasiregular maps.

However, the pseudosupremum does not solve the second major challenge surrounding \cite[Lemma 4.4]{Bonk-PoggiCorradini_Rickman-Picard}, which is the problem of showing that mappings with multiple quasiregular values in $\partial f(\R^n)$ satisfy $A_f(\R^n) = \infty$. We note that Theorem \ref{thm:Bonk-PoggiCorradini-part} is obtained by essentially ignoring this issue and instead assuming a-priori that $A_f(\R^n) = \infty$. For non-constant quasiregular maps $f \colon \R^n \to \S^n \setminus \{x_1, x_2\}$, the fact that $A_f(\R^n) = \infty$ follows easily; see for example \cite[Lemma~IV.2.7]{Rickman_book} or \cite[p.~631]{Bonk-PoggiCorradini_Rickman-Picard}. In our setting, however, this step becomes nontrivial, involving challenges somewhat similar to the ones encountered in the study of the Astala-Iwaniec-Martin uniqueness question. In particular, the part about excluding the case $A_f(\R^n) < \infty$ is the only part of the proof where the precise integrability assumptions of Theorems \ref{thm:Picard_for_QR_values} and \ref{thm:Picard_for_QR_values_spherical} become relevant.

\subsection{The structure of this paper}

In Section \ref{sect:sobolev_prelims}, we recall some preliminary information on Sobolev differential forms that is used in our computations. Section \ref{sect:qrvalues_into_Sn} is a discussion on the connections between the Euclidean and spherical definitions of quasiregular values. In Section \ref{sect:caccioppoli}, we prove the relevant Caccioppoli-type estimates that are used in the main proof.

With these preliminaries complete, we then prove Theorem \ref{thm:Bonk-PoggiCorradini-part} in Section~\ref{sect:main_proof}. The proof of Theorems \ref{thm:Picard_for_QR_values} and \ref{thm:Picard_for_QR_values_spherical} is then at last completed in Section \ref{sect:infinite_measure}, with the entire section dedicated to dealing with the special case where $A_f(\R^n) < \infty$. In Section \ref{sect:planar}, we prove the sharp planar result given in Theorem \ref{thm:q_is_2_when_n_is_2} by using Theorem \ref{thm:Picard_for_QR_values}. Finally, in Section \ref{sect:counterexamples}, we provide counterexamples which show the sharpness of the assumptions of Theorem \ref{thm:Picard_for_QR_values}.

\subsection{Acknowledgments}

We thank Pekka Pankka for several helpful comments and insights on the paper.

\section{Preliminaries on Sobolev differential forms}\label{sect:sobolev_prelims}

Throughout this paper, we use $C(a_1, a_2, \dots, a_m)$ to denote a positive constant that depends on the parameters $a_i$. The value of $C(a_1, a_2, \dots, a_m)$ may change in each estimate even if the parameters remain the same. We also use the shorthand $A_1 \lesssim_{a_1, a_2, \dots, a_m} A_2$ which stands for $A_1 \leq C(a_1, a_2, \dots, a_m) A_2$, where we always list the dependencies of the constant on the $\lesssim$-symbol. The shorthand $A_1 \gtrsim_{a_1, a_2, \dots, a_m} A_2$ is defined similarly. Additionally, if $B = \B^n(x, r) \subset \R^n$ is a Euclidean ball and $c \in (0, \infty)$, then we use $cB$ to denote the ball $\B^n(x, cr)$.

\vspace{\baselineskip}

Let $U$ be an open subset of $\R^n$. We use $L^p(\wedge^k T^*U)$, $L^p_\loc(\wedge^k T^*U)$, $W^{1,p}(\wedge^k T^*U)$, $W^{1,p}_\loc(\wedge^k T^* U)$, and $C^l(\wedge^k T^* U)$ to denote differential $k$-forms $\omega$ on $U$ such that the coefficients of $\omega$ with respect to the standard basis $\{dx_{i_1} \wedge \dots \wedge dx_{i_k} \mid i_1 < i_2 < \dots < i_k\}$ of $\wedge^k T^*\R^n$ are in $L^p(U)$, $L^p_\loc(U)$, $W^{1,p}(U)$, $W^{1,p}_\loc(U)$, or $C^l(U)$, respectively. We also use the sub-index $0$ to denote spaces of differential forms or real-valued functions with compact supports; for instance, $C^\infty_0(U)$ denotes the space of compactly supported smooth real-valued functions on $U$. 

Given a differential form $\omega \colon U \to \wedge^k T^*\R^n$, we use $\omega_x \in \wedge^k T^*_x \R^n$ to denote the value of $\omega$ at $x$. We use $\abs{\omega_x}$ for the norm of $\omega_x$, which is the $l^2$-norm on the coefficients of $\omega_x$ with respect to the standard basis; in particular $\abs{\omega}$ is a function $U \to [0, \infty)$. Recall that $\abs{\omega_1 \wedge \omega_2} \leq C(n) \abs{\omega_1} \abs{\omega_2}$, and if either $\omega_1$ or $\omega_2$ is a simple wedge product of $1$-forms, then one in fact has $\abs{\omega_1 \wedge \omega_2} \leq \abs{\omega_1} \abs{\omega_2}$. We also use $\hodge\omega$ to denote the Hodge star of a differential $k$-form $\omega$.

If $\omega \in L^1_\loc(\wedge^k T^* U)$, then $d\omega \in L^1_\loc(\wedge^{k+1} T^* U)$ is a \emph{weak differential} of $\omega$ if
\[
	\int_U d\omega \wedge \eta = (-1)^{k+1} \int_U \omega \wedge d\eta
\]
for every $\eta \in C^\infty_0(\wedge^{n-k-1} T^* U)$. We denote the space of $\omega \in L^p_\loc(\wedge^k T^* U)$ with a weak differential $d\omega \in L^q_\loc(\wedge^{k+1} T^* U)$ by $W^{d, p, q}_\loc(\wedge^k T^* U)$, with the abbreviation $W^{d, p}_\loc(\wedge^k T^* U) = W^{d, p, p}_\loc(\wedge^k T^* U)$. We also define versions with global integrability, denoted $W^{d, p, q}(\wedge^k T^* U)$ and $W^{d, p}(\wedge^k T^* U)$. Recall that $W^{1, p}_\loc(\wedge^k T^* U) \subset W^{d, p}_\loc(\wedge^k T^* U)$ and $W^{1,p}(\wedge^k T^* U) \subset W^{d, p}(\wedge^k T^* U)$, where the weak differential of an element of $W^{1, p}_\loc(\wedge^k T^* U)$ is obtained component-wise by the rule $d (f dx_{i_1} \wedge dx_{i_2} \wedge \dots \wedge dx_{i_k}) = df \wedge dx_{i_1} \wedge dx_{i_2} \wedge \dots \wedge dx_{i_k}$.

If $\omega_1 \in W^{1,p}_\loc(\wedge^k T^* U)$ and $\omega_2 \in W^{1,q}_\loc(\wedge^l T^* U)$ with $p^{-1} + q^{-1} = r^{-1} \leq 1$, then standard product rules of Sobolev functions yield that $\omega_1 \wedge \omega_2 \in W^{1, r}_\loc(\wedge^{k+l} T^* U)$, and
\begin{equation}\label{eq:leibniz_rule}
	d(\omega_1 \wedge \omega_2) = d\omega_1 \wedge \omega_2 + (-1)^k \omega_1 \wedge d\omega_2.
\end{equation} 
By using a convolution approximation argument, it can be shown that \eqref{eq:leibniz_rule} also holds if one instead assumes that $\omega_1 \in W^{d,p_1, q_1}_\loc(\wedge^k T^* U)$ and $\omega_2 \in W^{d,p_2, q_2}_\loc(\wedge^l T^* U)$ with $p_1^{-1} + p_2^{-1} = r^{-1} \leq 1$ and $\max(p_1^{-1} + q_2^{-1}, p_2^{-1} + q_1^{-1}) = s^{-1} \leq 1$, in which case $\omega_1 \wedge \omega_2 \in W^{d, r, s}_\loc(\wedge^{k+l} T^* U)$. Moreover, if $\omega \in W^{d, 1}_0(\wedge^{n-1} T^* U)$, then a convolution-based argument similarly yields
\begin{equation}\label{eq:d_integral_is_zero}
	\int_U d\omega = 0.
\end{equation}
We also note that if $\omega \in W^{d,1}_\loc(\wedge^k T^* U)$, then $d\omega \in W^{d, 1}_\loc(\wedge^{k+1} T^* U)$ with $dd\omega = 0$.

If $\omega \in C(\wedge^k T^* V)$, i.e.\ if the coefficients of $\omega$ are continuous, and if $f \in W^{1,n}_\loc(U, \R^n)$, then the pull-back $f^* \omega$ is well-defined and lies in $L^{n/k}_\loc(\wedge^k T^*U)$. We recall that in this case, we have the estimate
\begin{equation}\label{eq:pullback_norm_estimate}
	\abs{f^* \omega} \leq (\abs{\omega} \circ f) \abs{Df}^k.
\end{equation}
Indeed, if $\omega = \phi dx_{i_1} \wedge \dots \wedge dx_{i_k}$, then $\abs{f^* \omega} = \abs{(\phi \circ f) df_{i_1} \wedge \dots \wedge df_{i_k}} \leq (\abs{\phi} \circ f) \abs{Df}^k$, and the result for general $\omega$ then follows by Pythagoras.
 
Moreover, if instead $\omega \in C^1_0(\wedge^k T^* V)$, then it follows from the chain rule of $C^1_0$-functions and $W^{1,n}_\loc$-functions that $f^* \omega \in W^{d, n/k, n/(k+1)}_\loc(\wedge^k T^*U)$ and $d f^* \omega = f^* d \omega$; see e.g.\ the proof of \cite[Lemma 2.2]{Kangasniemi-Onninen_Heterogeneous}. We additionally note that if $f$ is also continuous, then the assumption $\omega \in C^1_0(\wedge^k T^* V)$ can be weakened to $\omega \in C^1(\wedge^k T^* V)$ by using smooth cutoff functions.

Generalizing the rule $d f^* \omega = f^* d \omega$ any further than this requires care, as  if $d\omega$ is not continuous, merely defining the form $f^* d \omega$ faces the challenge that the pull-back of a measurable form under a map $f \in C(U, V) \cap W^{1,n}_\loc(U, \R^n)$ might not even be well-defined. Regardless, the assumptions can be weakened to $\omega$ having locally Lipschitz coefficients, if one is careful in defining $f^* d\omega$. For us, it is enough to have the following fact about the existence of $d f^* \omega$, which follows in a straightforward manner from the chain rule for Lipschitz and Sobolev maps; see e.g.\ Ambrosio and Dal Maso \cite[Corollary 3.2]{Ambrosio-DalMaso}.

\begin{lemma}\label{lem:lip_form_pullback}
	Let $U , V \subset \R^n$ be open sets, let $f \in C(U, V) \cap W^{1,n}_\loc(U, \R^n)$, and let $\omega \in C(\wedge^k T^* V) \cap W^{1,\infty}_\loc(\wedge^k T^* V)$ for $k \in \{0, \dots, n-1\}$; i.e., we assume that $\omega$ has locally Lipschitz coefficients. Then $f^*\omega \in W^{d, n/k, n/(k+1)}_\loc(\wedge^k T^* U)$.
\end{lemma}

In particular, combining Lemma \ref{lem:lip_form_pullback} with \eqref{eq:leibniz_rule} unlocks the following tool.

\begin{cor}\label{cor:lip_form_pullback_wedges}
	Let $U , V \subset \R^n$ be open sets, let $f \in C(U, V) \cap W^{1,n}_\loc(U, \R^n)$, let $\omega_1 \in C(\wedge^k T^* V) \cap W^{1,\infty}_\loc(\wedge^k T^* V)$, and let $\omega_2 \in C(\wedge^l T^* V) \cap W^{1,\infty}_\loc(\wedge^l T^* V)$, with $k, l \in \Z_{\geq 0}$, $k + l \leq n-1$. Then 
	\[
		f^*(\omega_1 \wedge \omega_2) \in W^{d, \frac{n}{k+l}, \frac{n}{k+l+1}}_\loc(\wedge^{k+l} T^* U),
	\] 
	with
	\[
		d f^*(\omega_1 \wedge \omega_2) = (df^*\omega_1) \wedge \omega_2 + (-1)^k f^* \omega_1 \wedge (d f^* \omega_2).
	\]
\end{cor}

\section{Quasiregular values and maps between spheres}\label{sect:qrvalues_into_Sn}

\subsection{Maps into $\S^n$}
Let $e_i$ denote the standard basis vectors of $\R^n$, let $\ip{\cdot}{\cdot}$ denote the Euclidean inner product on $\R^n$, and let $\abs{\cdot}$ denote the induced norm. The $n$-dimensional unit sphere $\S^n$ consists of all $w \in \R^{n+1}$ with $\abs{w} = 1$. Recall that on $\R^n$, the inverse $s_n \colon \R^n \to \S^n \setminus \{-e_{1}\}$ of the stereographic projection is defined by
\[
	s_n(x) = \frac{1}{1 + \abs{x}^2} \left(1-\abs{x}^2, 2x_1, 2x_2, \dots, 2x_n \right).
\]
The map $s_n$ is then extended to $\R^n \cup \{\infty\}$ by setting $s_n(\infty) = -e_{1}$. 

We recall that the \emph{spherical distance} $\sigma$ on $\S^n$ is given by $\sigma(w_1, w_2) = \arccos \ip{w_1}{w_2}$ for $w_1, w_2 \in \S^n$, using the embedding of $\S^n$ into $\R^{n+1}$. We also define the spherical distance on $\R^n \cup \{\infty\}$ by setting $\sigma(x_1, x_2) = \sigma(s_n(x_1), s_n(x_2))$ for $x_1, x_2 \in \R^n \cup \{\infty\}$. Via an elementary computation, one sees for $x_1, x_2 \in \R^n$ that
\[
	\cos(\sigma(x_1, x_2)) = \ip{s_n(x_1)}{s_n(x_2)} = 1 - \frac{2\abs{x_1 - x_2}^2}{(1 + \abs{x_1}^2)(1 + \abs{x_1}^2)}.
\]
In particular, 
\begin{equation}\label{eq:spherical_metric_conversion}
	\sin \frac{\sigma(x_1, x_2)}{2} = \frac{\abs{x_1 - x_2}}{\sqrt{(1 + \abs{x_1}^2)(1 + \abs{x_2}^2)}} \qquad \text{for } x_1, x_2 \in \R^n.
\end{equation}
By letting $x_2$ tend to infinity in \eqref{eq:spherical_metric_conversion}, we also see that
\begin{equation}\label{eq:spherical_metric_conversion_infty}
	\sin \frac{\sigma(x_1, \infty)}{2} = \frac{1}{\sqrt{1 + \abs{x_1}^2}} \qquad \text{for } x_1 \in \R^n.
\end{equation}

We equip $\S^n$ with the standard Riemannian metric that arises from the embedding to $\R^{n+1}$, and orient $\S^n$ so that its volume form $\vol_{\S^n}$ is given by the restriction of the $n$-form $\hodge d (2^{-1}\abs{x}^2) \in C^\infty(\wedge^n T^*\R^{n+1})$. When $\S^n$ is equipped with this metric and volume form, the map $s_n \colon \R^n \to \S^n$ is conformal; more precisely, 
\begin{equation}\label{eq:stero_proj_conformal}
	\abs{Ds_n(x)}^n = J_{s_n}(x) = \frac{2^n}{\bigl(1 + \abs{x}^2\bigr)^n}
\end{equation}
for every $x \in \R^n$. Moreover, given a set $U \subset \R^n \cup \{\infty\}$, we denote its spherical measure by $\vol_{\S^n}(U)$. By using the formula \eqref{eq:stero_proj_conformal} for $J_{s_n}$, we see that
\begin{equation}\label{eq:spherical_measure_as_integral}
	\vol_{\S^n}(U) = \int_U \frac{2^n \vol_n}{\bigl(1 + \abs{x}^2\bigr)^n}.
\end{equation}

Suppose then that $f \in W^{1,n}_\loc(\R^n, \R^n)$. We define a measurable map $h \colon \R^n \to \S^n$ by $h = s_n \circ f$. Since $s_n \colon \R^n \to \R^{n+1}$ is a smooth Lipschitz map, it follows that $h \in W^{1,n}_\loc(\R^n, \R^{n+1})$, and moreover $Dh(x) = Ds_n(f(x)) Df(x)$ for a.e.\ $x \in \R^n$. In particular, the image of $Dh(x)$ lies in $T_{h(x)} \S^n$ for a.e.\ $x$, and hence $Dh$ can be understood as a measurable map $T\R^n \to T\S^n$. Consequently, we obtain a Jacobian of $h$ by $J_h \vol_n = h^* \vol_{\S^n}$. Since $s_n$ is conformal, we obtain by \eqref{eq:stero_proj_conformal} that
\begin{equation}\label{eq:Dh_and_Jh}
	\abs{Dh}^n = \frac{2^n \abs{Df}^n}{\bigl(1 + \abs{f}^2 \bigr)^n}
	\qquad \text{and} \qquad 
	J_h = \frac{2^n J_f}{\bigl(1 + \abs{f}^2 \bigr)^n}
\end{equation}
a.e.\ in $\R^n$.

We then recall that we have given a definition of quasiregular values with respect to the Euclidean metric \eqref{eq:QRvalue_def} and with respect to the spherical metric \eqref{eq:spherical_QR_value}. We now prove comparison results for these two definitions. We begin with the spherical interpretation of Euclidean quasiregular values.

\begin{lemma}\label{lem:QRvalue_sphere}
	Let $f \in W^{1,n}_\loc(\Omega, \R^n)$ with $\Omega \subset \R^n$. Let $h = s_n \circ f$, let $w_0 = s_n(y_0)$ for some $y_0 \in \R^n$, let $K \in \R$, and let $\Sigma \colon \Omega \to [0, \infty)$ be measurable. Then the following conditions are equivalent up to an extra constant factor $C = C(n, y_0)$ on $\Sigma$:
	\begin{enumerate}
		\item\label{enum:qrval_eucl} $f$ has a $(K, \Sigma)$-quasiregular value at $y_0$;
		\item\label{enum:qrvals_spheric_separate} $h$ has a $(K, \Sigma)$-quasiregular value with respect to the spherical metric at both $w_0$ and $s_n(\infty)$;
		\item\label{enum:qrvals_spheric_combined} $h$ satisfies
		\begin{equation*}
			\abs{Dh}^n \leq K J_h + \sigma^n(h, w_0)  \sigma^n(h, s_n(\infty)) \Sigma
		\end{equation*}
		a.e.\ in $\Omega$.
	\end{enumerate}
\end{lemma}
\begin{proof}
	We first show the (almost) equivalence of \eqref{enum:qrval_eucl} and \eqref{enum:qrvals_spheric_combined}. We multiply \eqref{eq:QRvalue_def} on both sides by $2^n (1 + \abs{f}^2)^{-n}$ and use \eqref{eq:Dh_and_Jh}, obtaining that \eqref{eq:QRvalue_def} is equivalent to
	\[
		\abs{Dh}^n \leq K J_h + 2^n \frac{\abs{f - y_0}^n}{\bigl( 1 + \abs{f}^2\bigr)^n} \Sigma.
	\]
	Now, using \eqref{eq:spherical_metric_conversion} and \eqref{eq:spherical_metric_conversion_infty}, we observe that
	\begin{align*}
		\frac{\abs{f - y_0}}{1 + \abs{f}^2}
		&= \frac{\abs{f - y_0}}{\sqrt{(1 + \abs{f}^2)( 1 + \abs{y_0}^2)}} \cdot \frac{1}{\sqrt{1 + \abs{f}^2}} \cdot \sqrt{1 + \abs{y_0}^2}\\
		&= \sin \frac{\sigma(f, y_0)}{2} \cdot \sin \frac{\sigma(f, \infty)}{2} \cdot \left( \sin \frac{\sigma(y_0, \infty)}{2} \right)^{-1}.
	\end{align*}
	Thus, \eqref{eq:QRvalue_def} is equivalent to
	\[
		\abs{Dh}^n \leq K J_h + \frac{2^n \sin^n \bigl(2^{-1}\sigma(f, y_0)\bigr) \sin^n \bigl(2^{-1}\sigma(f, \infty)\bigr)}{\sin^n(2^{-1} \sigma(y_0, \infty))} \Sigma.
	\]
	Since $(2/\pi) t \leq \sin(t) \leq t$ whenever $t \in [0, \pi/2]$, the previous equation is equivalent to the one in part \eqref{enum:qrvals_spheric_combined}, up to an extra constant on $\Sigma$.
	
	It remains to show the (almost) equivalence of \eqref{enum:qrvals_spheric_separate} and \eqref{enum:qrvals_spheric_combined}. Since $\sigma(\cdot, \cdot)$ is bounded from above by $\pi$, it is clear from the definition of spherical quasiregular values in \ref{eq:spherical_QR_value} that \eqref{enum:qrvals_spheric_combined} implies \eqref{enum:qrvals_spheric_separate} up to an extra factor of $\pi^n$ on $\Sigma$. For the other diection, we use the fact that for any distinct $w_1, w_2 \in \S^n$, the function $w \mapsto \min (\sigma^{-1}(w, w_1), \sigma^{-1}(w, w_2))$ is continuous and has a maximum value of $2/\sigma(w_1, w_2)$. Thus, if we have \eqref{enum:qrvals_spheric_separate}, then we have the estimate
	\begin{align*}
		\abs{Dh}^n &\leq K J_h 
		+ \min\bigl(\sigma^n(h, w_0), \sigma^n(h, s_n(\infty))\bigr) \Sigma\\
		&= K J_h 
		+ \min\bigl(\sigma^{-n}(h, s_n(\infty)), \sigma^{-n}(h, w_0)\bigr) 
			\sigma^n(h, w_0) \sigma^n(h, s_n(\infty)) \Sigma\\
		&\leq K J_h + C(n, y_0) \sigma^n(h, w_0) \sigma^n(h, s_n(\infty)) \Sigma
	\end{align*}
	a.e.\ on $\Omega$, completing the proof.
\end{proof}

Next, we give the Eculidean interpretation of spherical quasiregular values.

\begin{lemma}\label{lem:QRvalue_sphere_single}
	Let $f \in W^{1,n}_\loc(\Omega, \R^n)$ with $\Omega \subset \R^n$. Let $h = s_n \circ f$, let $w_0 = s_n(y_0)$ for some $y_0 \in \R^n$, let $K \in \R$, and let $\Sigma \colon \Omega \to [0, \infty)$ be measurable. Then the following conditions are equivalent up to an extra constant factor $C = C(n, y_0)$ on $\Sigma$:
	\begin{enumerate}
		\item\label{enum:qrval_spheric} $h$ has a $(K, \Sigma)$-quasiregular value at $w_0$;
		\item\label{enum:qrval_eucl_weak} $f$ satisfies
		\begin{equation*}
			\abs{Df}^n \leq K J_f + \abs{f - y_0}^n \bigl( 1 + \abs{f}^2 \bigr)^\frac{n}{2} \Sigma
		\end{equation*}
		a.e.\ in $\Omega$.
	\end{enumerate}
	Similarly, the following conditions are equivalent up to an extra constant factor $C = C(n)$ on $\Sigma$:
	\begin{enumerate}[label=(\arabic*')]
		\item\label{enum:qrval_spheric_infty} $h$ has a $(K, \Sigma)$-quasiregular value at $s_n(\infty)$;
		\item\label{enum:qrval_eucl_weak_infty} $f$ satisfies
		\begin{equation*}
			\abs{Df}^n \leq K J_f + \bigl( 1 + \abs{f}^2 \bigr)^\frac{n}{2} \Sigma
		\end{equation*}
		a.e.\ in $\Omega$.
	\end{enumerate}
\end{lemma}
\begin{proof}
	For the first equivalence, similarly as in the proof of Lemma \ref{lem:QRvalue_sphere}, we may use \eqref{eq:spherical_metric_conversion}, \eqref{eq:spherical_metric_conversion_infty}, and \eqref{eq:Dh_and_Jh} to show that condition \eqref{enum:qrval_eucl_weak} is equivalent to 
	\[
		\abs{Dh}^n \leq K J_h + \frac{2^n \sin^n \bigl(2^{-1}\sigma(h, w_0)\bigr)}{\sin^n(2^{-1} \sigma(w_0, s_n(\infty)))} \Sigma.
	\]
	Since this is equivalent to \eqref{eq:spherical_QR_value} up to a constant of comparison on $\Sigma$, the claim follows. The proof of the second equivalence is analogous, as \eqref{eq:spherical_metric_conversion_infty} and \eqref{eq:Dh_and_Jh} yield that condition \ref{enum:qrval_eucl_weak_infty} is equivalent to
	\[
		\abs{Dh}^n \leq K J_h + 2^n \sin^n \biggl( \frac{\sigma(h, s_n(\infty))}{2} \biggr) \Sigma.
	\]
\end{proof}

We end this section by pointing out that the single-point Liouville theorem and Reshetnyak's theorem for Euclidean quasiregular values imply corresponding results for spherical quasiregular values.

\begin{prop}\label{prop:Liouville_and_Reshetnyak_spherical}
	Let $\Omega \subset \R^n$ be a domain, let $\eps > 0$, and let $h \in W^{1,n}_\loc(\Omega, \S^n)$ be a continuous map with a quasiregular value with respect to the spherical metric at $w_0 \in \S^n$, for given choices of $K \geq 1$ and $\Sigma \colon \Omega \to [0, \infty)$. Then the following results hold.
	\begin{enumerate}[label=(\roman*)]
		\item \label{item:Reshetnyak_sphere} (Reshetnyak's theorem) If $\Sigma \in L^{1+\eps}_\loc(\Omega)$ and if $h$ is not the constant function $h \equiv w_0$, then $h^{-1}\{w_0\}$ is discrete, the local index $i(x, h)$ is positive at every $x \in h^{-1}\{w_0\}$, and $h$ maps every neighborhood $U \subset \Omega$ of a point of $f^{-1}\{w_0\}$ to a neighborhood $h(U)$ of $w_0$.
		\item \label{item:Liouville_sphere} (Liouville theorem) If $\Omega = \R^n$, $\Sigma \in L^{1+\eps}_\loc(\R^n) \cap L^1(\R^n)$, and $\overline{h(\R^n)} \neq \S^n$, then either $h \equiv w_0$ or $w_0 \notin h(\R^n)$.
	\end{enumerate}
\end{prop}
\begin{proof}
	Suppose first that the assumptions of \ref{item:Reshetnyak_sphere} hold. If we post-compose $h$ with an isometric rotation $R \colon \S^n \to \S^n$, it follows that $R \circ h$ has a $(K, \Sigma)$- quasiregular value with respect to the spherical metric at $R(w_0)$. Thus, we may assume that $w_0 \neq \infty$. Let $y_0 \in \R^n$ be the point for which $s_n(y_0) = w_0$.
	
	We select an open neighborhood $U$ of $w_0$ such that $\infty \notin \overline{U}$. Now, in the set $\Omega' = h^{-1} U$, there is a bounded, continuous $f \in W^{1,n}_\loc(\Omega', \R^n)$ such that $h = s_n \circ f$. By Lemma \ref{lem:QRvalue_sphere_single}, $f$ has a $(K, \Sigma')$-quasiregular value at $y_0$, where $\Sigma' = C(n, y_0) (1 + \abs{f}^2)^{n/2} \Sigma$. By boundedness of $f$, clearly $\Sigma' \in L^{1+\eps}_\loc(\Omega')$. Now, the Euclidean result (Theorem \ref{thm:single_value_results} \ref{item:Reshetnyak}) yields the claim for $h \vert_{\Omega'}$. Since $\Omega'$ is the pre-image of a neighborhood of $w_0$ under $h$, this in fact implies the result for $h$.
	
	Suppose then that the assumptions of \ref{item:Liouville_sphere} hold. If $w_0 \in \S^n \setminus \overline{h(\R^n)}$, then clearly $w_0 \notin h(\R^n)$ and the claim holds. Otherwise, by post-composing by an isometric rotation, we may this time assume that $s_n(\infty) \in \S^n \setminus \overline{h(\R^n)}$ and that $w_0 = s_n(y_0)$ for some $y_0 \in \R^n$. Consequently, we obtain a bounded, continuous map $f \in W^{1,n}_\loc(\R^n, \R^n)$ such that $h = s_n \circ f$. Lemma \ref{lem:QRvalue_sphere_single} again yields that $f$ has a $(K, \Sigma')$-quasiregular value at the point $y_0$, where $\Sigma' = C(n, y_0) (1 + \abs{f}^2)^{n/2} \Sigma$. By boundedness of $f$, we have $\Sigma' \in L^1(\R^n) \cap L^{1+\eps}_\loc(\R^n)$, and hence the corresponding Euclidean result (Theorem \ref{thm:single_value_results} \ref{item:Liouville}) implies that either $w_0 \notin h(\R^n)$ or $h \equiv w_0$.
\end{proof}

\section{Logarithmic singularity and Caccioppoli inequalities}\label{sect:caccioppoli}

In this section, we prove the Caccioppoli-type inequalitied used in the proof. In particular, we require counterparts to \cite[Lemmas 4.2 and 5.4]{Bonk-PoggiCorradini_Rickman-Picard} where we assume \eqref{eq:QRvalue_def} instead of full quasiregularity. Since our setting still allows for large sets where $J_f(x) = 0$ and $Df(x)$ is non-invertible, we lack a good counterpart for the induced conformal structure $G_f(x) = J_f^{-2/n}(x) [D^{T}\!f(x) \, Df(x)]^{-1}$ of a quasiregular map. Thus, instead of using $\cA$-subharmonic theory as in the original proofs, we rely on more direct computations.

\subsection{The logarithmic singularity function}

We begin by recalling the logarithmic singularity function from \cite[Section 3]{Bonk-PoggiCorradini_Rickman-Picard}. We first define a function $S \colon [0, \infty) \to [0, 1)$ by
\begin{equation}\label{eq:S_def}
	S(r) = \frac{\vol_{\S^n}(\B^n(0, r))}{\vol_{\S^n}(\R^n)}.
\end{equation}
By using \eqref{eq:spherical_measure_as_integral}, one can see that
\[
	S(r) = \frac{2^n \vol_{\S^{n-1}}(\R^{n-1})}{\vol_{\S^n}(\R^n)} \int_0^r \frac{t^{n-1} \dd t}{\bigl(1 + t^2\bigr)^n}.
\]
In particular,
\begin{equation}\label{eq:S_deriv}
	S'(r) = \frac{C(n) r^{n-1}}{\bigl(1 + r^2\bigr)^n},
\end{equation}
and we obtain the the following estimates describing the asymptotic behavior of $S(r)$ and $S'(r)$ for large and small $r$:
\begin{equation}\label{eq:S_estimates}
	S(r) \lesssim_n \min(r^n, 1), \qquad S'(r) \lesssim_n \min(r^{n-1}, r^{-(n+1)}).
\end{equation}
Next, a function $H \colon [0, \infty) \to [0, \infty)$ is defined by
\begin{equation}\label{eq:H_def}
	H(r) = \int_0^r \frac{S^\frac{1}{n-1}(t) \dd t}{t}.
\end{equation}
Consequently, we have
\begin{equation}\label{eq:H_deriv}
	H'(r) = \frac{S^\frac{1}{n-1}(r)}{r},
\end{equation}
and by applying \eqref{eq:S_estimates}, we get the estimates
\begin{equation}\label{eq:H_estimates}
	H(r) \lesssim_n \min\left(r^\frac{n}{n-1}, 1 + \abs{\log(r)}\right), \qquad H'(r) \lesssim_n \min\left(r^\frac{1}{n-1}, r^{-1}\right).
\end{equation}

The logarithmic potential $v \colon \R^n \to [0, \infty]$ at infinity is then defined on $\R^n$ by
\begin{equation}\label{eq:v_def}
	v(x) = H(\abs{x}).
\end{equation}
Since $v$ is a real-valued radial function, we have
\begin{align}\label{eq:v_abs_value}
	\nabla v(x) = H'(\abs{x}) \frac{x}{\abs{x}} \qquad \text{and}\qquad \abs{\nabla v(x)} = H'(\abs{x}).
\end{align}
Moreover, recall that the \emph{$n$-Laplacian} $\Delta_n v$ of $v$ is equivalently defined by either of the following two formulae:
\[
	\Delta_n v = \nabla \cdot (\abs{\nabla v}^{n-2} \nabla v), \qquad (\Delta_n v) \vol_n = d(\abs{dv}^{n-2} \hodge dv).
\]
We record that the $n$-Laplacian of $v$ is in fact exactly the density of the spherical volume; we refer to \cite[Lemma 3.1]{Bonk-PoggiCorradini_Rickman-Picard} for the proof. 

\begin{lemma}\label{lem:v_properties}
	Let $v \colon \R^n \to [0, \infty)$ be as in \eqref{eq:v_def}. Then $v \in C^1(\R^n)$, $\abs{\nabla v}^{n-2} \nabla v \in C^1(\R^n, \R^n)$, and 
	\[
		\Delta_n v(x) = \frac{2^n}{\bigl(1 + \abs{x}^2\bigr)^n} = J_{s_n}(x).
	\]  
\end{lemma}

\subsection{Quasiregular values and superlevel sets}

The use of sublevel and superlevel sets has been perhaps the most fundamental tool in obtaining the current results on quasiregular values; see \cite[Section 5]{Kangasniemi-Onninen_Heterogeneous} and \cite[Section 4]{Kangasniemi-Onninen_1ptReshetnyak}. They also play a key role in this paper. Indeed, we require a counterpart to \cite[Lemma 4.4]{Bonk-PoggiCorradini_Rickman-Picard}, which essentially yields that the superlevel sets $\{\abs{f} > L\}$ of a non-constant entire quasiregular function $f$ have no bounded components. As stated in the introduction, superlevel set methods do not fully eliminate the existence of bounded components of $\{\abs{f} > L\}$ in our case, which ends up causing significant complications during the proof. However, we do get a type of control on the total size of any bounded components of $\{\abs{f} > L\}$.

In particular, our main counterpart to \cite[Lemma 4.4]{Bonk-PoggiCorradini_Rickman-Picard} is the following general result, which is similar in spirit to \cite[Lemma 5.3]{Kangasniemi-Onninen_Heterogeneous} and \cite[Lemma~4.3]{Kangasniemi-Onninen_1ptReshetnyak}.

\begin{lemma}\label{lem:superlevel_bdd_component}
	Let $y_0 \in \R^n$ and $r > 0$. Suppose that $f \in W^{1,n}_\loc(\R^n, \R^n)$ is continuous and satisfies an estimate of the form
	\begin{equation}\label{eq:general_sigma_ineq}
		\abs{Df}^n \leq K J_f + \tilde{\Sigma},
	\end{equation}
	where we assume $K \in \R$ and $\tilde{\Sigma} \in L^{1}_\loc(\R^n)$. Let $U$ be a bounded component of $f^{-1}(\R^n \setminus \overline{\B^n}(y_0, r))$. Then for any continuous function $\Phi \colon [r, \infty) \to [0, \infty)$, we have
	\[
		\int_U \Phi(\abs{f - y_0}) \abs{Df}^n \leq \int_U \Phi(\abs{f - y_0}) \tilde{\Sigma}.
	\]
\end{lemma}
\begin{proof}
	Since $U$ is bounded and since $f$ is continuous, $f(\overline{U})$ is compact, and hence there exists $y_1 \in \R^n \setminus (\overline{\B^n}(y_0, r) \cup f(\overline{U}))$. By the boundedness of $U$ and the continuity of $f$ and $\Phi$, we also have that the functions $\Phi(\abs{f - y_0}) \abs{Df}^n$ and $\Phi(\abs{f - y_0}) \tilde{\Sigma}$ are integrable over $U$.  By a Sobolev change of variables, see e.g.\ \cite[Theorem 5.27]{Fonseca-Gangbo-book}, we have
	\[
		\int_U \Phi(\abs{f - y_0}) J_f = \int_{\R^n \setminus \overline{\B^n}(y_0, r)} \Phi(\abs{y - y_0}) \deg(f, y, U) \vol_n(y). 
	\]
	However, since $U$ is a connected component of $f^{-1}(\R^n \setminus \overline{\B^n}(y_0, r))$, we have $f(\partial U) \subset \partial \B^n(y_0, r)$, and since also $\R^n \setminus \overline{\B^n}(y_0, r)$ is connected, we have $\deg(f, y, U) = \deg(f, y_1, U) = 0$ for every $y \in \R^n \setminus \overline{\B^n}(y_0, r)$; see for instance \cite[Theorem 2.1 and Theorem 2.3 (3)]{Fonseca-Gangbo-book}. In conclusion,
	\[
		\int_U \Phi(\abs{f - y_0}) J_f = 0.
	\]
	Consequently, the desired estimate follows by multiplying \eqref{eq:general_sigma_ineq} by $\Phi(\abs{f - y_0})$ and by integrating both sides over $U$.
\end{proof}

\subsection{Measure estimates and Caccioppoli-type inequalities}

We then let $K \geq 1$ and $\Sigma \in L^1(\R^n) \cup L^{1+\eps}_\loc(\R^n)$ for some $\eps > 0$, and suppose that $f \in W^{1,n}_\loc(\R^n, \R^n)$ is a continuous map such that $s_n \circ f$ has a $(K, \Sigma)$-quasiregular value with respect to the spherical metric at $s_n(\infty)$. Note that by Lemma \ref{lem:QRvalue_sphere}, this assumption is true if $f$ has a $(K, \Sigma)$-quasiregular value at any point $y_0 \in \R^n$, up to an additional multiplicative constant $C = C(n, y_0)$ on $\Sigma$. Moreover, by Lemma \ref{lem:QRvalue_sphere_single}, the map $f$ satisfies 
\begin{equation}\label{eq:QRvalue_at_infty_eucl}
		\abs{Df}^n \leq K J_f + C(n) \bigl(1 + \abs{f}^2\bigr)^{\frac{n}{2}} \Sigma
\end{equation}
a.e.\ in $\R^n$. 

We use the notation $\Sigma(E)$ to denote the integral of $\Sigma$ over a measurable set $E \subset \R^n$. We then let $A_f$ denote the pull-back of the spherical measure under $f$. In particular,
\begin{equation}\label{eq:Af_def}
	A_f(E) = \int_E \frac{2^n J_f}{\bigl(1 + \abs{f}^2\bigr)^n}.
\end{equation}
Note that unlike in the quasiregular case, our $A_f$ is a signed measure. We use $A_f^{-}$ and $A_f^{+}$ to denote its positive and negative parts, and $\abs{A_f}$ to denote its total variation measure. We first use our assumption that $\Sigma \in L^1(\R^n)$ to show that $A_f$ is well defined, i.e.\ that we cannot have $A_f^{+}(E) = A_f^{-}(E) = \infty$. Indeed, \eqref{eq:QRvalue_at_infty_eucl} can be rewritten as $\abs{Df}^n + K J_f^{-} \leq K J_f^{+} + C(n) (1 + \abs{f}^2)^{n/2} \Sigma$, and since $J_f^{+}$ vanishes when $J_f^{-}$ is non-zero, we hence obtain
\begin{equation}\label{eq:QRvalue_Jfminus}
	J_f^{-} \lesssim_n K^{-1} (1 + \abs{f}^2)^{n/2} \Sigma.
\end{equation}
In particular, using $K^{-1} \leq 1$, \eqref{eq:QRvalue_Jfminus} yields the estimate
\begin{equation}\label{eq:Afminus_estimate}
	A_f^{-}(E) = \int_{E} \frac{2^n J_f^{-}}{\bigl(1 + \abs{f}^2\bigr)^n} \lesssim_n \int_{E} \frac{2^n}{K \bigl(1 + \abs{f}^2\bigr)^\frac{n}{2}} \Sigma \lesssim_n \Sigma(E).
\end{equation}
Since $\Sigma(\R^n) < \infty$, we hence have $A_f^{-}(\R^n) \leq C(n) \Sigma(\R^n) < \infty$, and thus $A_f$ is a well defined signed measure on all measurable subsets of $\R^n$.

With the measure $A_f$ defined, we start with a technical Caccioppoli-type estimate that sees multiple uses in the proofs.

\begin{lemma}\label{lem:sublevel_Caccioppoli}
	Let $f \in W^{1,n}_\loc(\R^n, \R^n)$ be a non-constant continuous function that satisfies \eqref{eq:QRvalue_at_infty_eucl}, where $K \geq 1$ and $\Sigma \in L^1(\R^n) \cap L^{1+\eps}_\loc(\R^n)$ for some $\eps > 0$. Let $u = v \circ f$, where $v$ is as in \eqref{eq:v_def}. Then for every $L \geq 0$ and $\eta \in C^\infty_0(\R^n)$ with $\eta \geq 0$, we have
	\begin{multline*}
	\int_{\{u < L\}} \eta^n \bigl(\abs{dv}^n \circ f \bigr) \abs{Df}^n\\
	\lesssim_{n} K^{n} L^n \int_{\{u < L\}} \abs{d\eta}^n 
	+ K L \int_{\{u < L\}} \eta^n \dd \smallabs{A_f}
	+ C(n) \int_{\{u < L\}} \eta^n \Sigma.
	\end{multline*}
\end{lemma}
\begin{proof}
	We denote $v_L = \min(v, L)$ and $u_L = \min(u, L) = v_L \circ f$. We may assume $L > 0$, as the case $L = 0$ is trivial due to $\{u < L\}$ being empty in this case.
	
	We first observe that $(\abs{dv} \circ f) (1 + \abs{f}^2)^{1/2} = H'(\abs{f}) (1 + \abs{f}^2)^{1/2} \leq C(n)$ by \eqref{eq:v_abs_value} and \eqref{eq:H_estimates}. We combine this with \eqref{eq:QRvalue_at_infty_eucl}, obtaining
	\begin{multline}\label{eq:dv_DF_estimate_part_1}
		\int_{\{u < L\}} \eta^n \bigl(\abs{dv}^n \circ f \bigr) \abs{Df}^n\\
		\leq K \int_{\{u < L\}} \eta^n \bigl(\abs{dv}^n \circ f \bigr) J_f 
		+ C(n) \int_{\{u < L\}} \eta^n \bigl(\abs{dv}^n \circ f \bigr) \bigl(1 + \abs{f}^2\bigr)^{\frac{n}{2}} \Sigma\\
		\leq K \int_{\{u < L\}} \eta^n f^*(\abs{dv}^n \vol_n)  
		+ C(n) \int_{\{u < L\}} \eta^n \Sigma.
	\end{multline}
	Let then $\cX_{\{u < L\}}$ be the characteristic function of $\{u < L\}$. We claim that
	\begin{equation}\label{eq:technical_truncation_eq}
		 \cX_{\{u < L\}} f^*(\abs{dv}^n \vol_n) = du_L \wedge f^*(\abs{dv}^{n-2} \hodge dv)
	\end{equation}
	a.e.\ in $\R^n$. Indeed, $du_L$ vanishes a.e.\ in the set $\{u \geq L\}$ by e.g.\ \cite[Corollary 1.21]{Heinonen-Kilpelainen-Martio_book}, and in $\{u < L\}$, we may compute as follows: $du_L \wedge f^*(\abs{dv}^{n-2} \hodge dv) = f^*dv \wedge f^*(\abs{dv}^{n-2} \hodge dv) = f^*(\abs{dv}^{n-2} dv \wedge \hodge dv) = f^*(\abs{dv}^n \vol_n)$. 
	
	We then observe that $(v_L - L) \abs{dv}^{n-2} \hodge dv$ has Lipschitz coefficients. Hence, by Lemma \ref{lem:lip_form_pullback}, $f^* ((v_L - L) \abs{dv}^{n-2} \hodge dv) \in W^{d, n/(n-1), 1}_\loc(\wedge^{n-1} T^* \R^n)$, and by Corollary \ref{cor:lip_form_pullback_wedges},
	\begin{multline*}
		d f^* ((v_L - L) \abs{dv}^{n-2} \hodge dv)\\
		= du_L \wedge f^*(\abs{dv}^{n-2} \hodge dv) + (u_L - L) d f^*(\abs{dv}^{n-2} \hodge dv)
	\end{multline*}
	Now, by using \eqref{eq:d_integral_is_zero}, we may compute that
	\begin{multline}\label{eq:d_caccioppoli}
		\int_{\R^n} \eta^n du_L \wedge f^*(\abs{dv}^{n-2} \hodge dv)\\
		= \int_{\R^n} \eta^n \left[ df^*((v_L - L)\abs{dv}^{n-2} \hodge dv) - (u_L - L) df^*(\abs{dv}^{n-2} \hodge dv) \right]\\
		= -\int_{\R^n} d\eta^n \wedge f^*((v_L - L)\abs{dv}^{n-2} \hodge dv) - \int_{\R^n} \eta^n (u_L - L) df^*(\abs{dv}^{n-2} \hodge dv).
	\end{multline} 
	
	By Lemma \ref{lem:v_properties}, we have that $\abs{dv}^{n-2} \hodge dv$ is a $C^1$-smooth form, and consequently $df^*(\abs{dv}^{n-2} \hodge dv) = f^*d(\abs{dv}^{n-2} \hodge dv) = f^* s_n^* \vol_{\S^n}$ weakly. On the other hand, by using \eqref{eq:pullback_norm_estimate}, we obtain
	\[
		\big\lvert f^*((v_L - L)\abs{dv}^{n-2} \hodge dv) \big\rvert \leq \abs{u_L - L} \bigl(\abs{dv}^{n-1} \circ f\bigr) \abs{Df}^{n-1}.
	\]
	By combining these computations with \eqref{eq:technical_truncation_eq} and \eqref{eq:d_caccioppoli}, we hence obtain the estimate
	\begin{multline*}
		K\int_{\{u < L\}} \eta^n f^*(\abs{dv}^n \vol_n) = K\int_{\R^n} \eta^n du_L \wedge f^*(\abs{dv}^{n-2} \hodge dv)\\
		\leq K\int_{\R^n} \abs{d \eta^n} \abs{f^* ((v_L - L)\abs{dv}^{n-2} \hodge dv)} + K \int_{\R^n} \abs{u_L - L} \eta^n \abs{f^* s_n^* \vol_{\S^n}}\\
		\leq Kn \int_{\R^n} \abs{u_L - L} \abs{d \eta} \left( \eta (\abs{dv} \circ f) \abs{Df}\right)^{n-1} + K\int_{\R^n} \eta^n \abs{u_L - L} \dd \smallabs{A_f}.
	\end{multline*}
	Moreover, since $u_L - L = 0$ outside $\{u < L\}$, and since $\abs{u_L - L} \leq L$, we in fact get
	\begin{multline}\label{eq:dv_DF_estimate_part_2}
		K\int_{\{u < L\}} \eta^n f^*(\abs{dv}^n \vol_n)\\
		\leq KLn \int_{\{u < L\}}  \abs{d \eta} \left( \eta (\abs{dv} \circ f) \abs{Df}\right)^{n-1} + KL\int_{\{u < L\}} \eta^n \dd \smallabs{A_f}.
	\end{multline}

	We recall Young's inequality, which states that $ab \leq a^p/p + b^q/q$ for $a, b \geq 0$ and $p, q \geq 1$ with $p^{-1} + q^{-1} = 1$. We estimate the first term of the right hand side of \eqref{eq:dv_DF_estimate_part_2} by Young's inequality, resulting in
	\begin{multline}\label{eq:dv_dF_estimate_part_3}
		KLn\int_{\{u < L\}} \abs{d \eta} \left( \eta (\abs{dv} \circ f) \abs{Df}\right)^{n-1}\\
		\leq K^{n} L^n n^{n-1} \int_{\{u < L\}} \abs{d \eta}^n + \frac{n-1}{n} \int_{\{u < L\}} \eta^n \bigl(\abs{dv}^n \circ f \bigr) \abs{Df}^n.
	\end{multline}
	We note that since $\abs{dv}$ is bounded by \eqref{eq:v_abs_value} and \eqref{eq:H_estimates}, $\eta^n (\abs{dv}^n \circ f) \abs{Df}^n$ has finite integral over $\R^n$. We hence chain \eqref{eq:dv_DF_estimate_part_1}, \eqref{eq:dv_DF_estimate_part_2}, and \eqref{eq:dv_dF_estimate_part_3} together, and absorb the integral of $\eta^n (\abs{dv}^n \circ f) \abs{Df}^n$ from the right side of \eqref{eq:dv_dF_estimate_part_3} to the left side of \eqref{eq:dv_DF_estimate_part_1}. The claim follows.
\end{proof}

The most immediate consequence of Lemma \ref{lem:sublevel_Caccioppoli} is the following corollary, which is our counterpart to \cite[Lemma 5.4]{Bonk-PoggiCorradini_Rickman-Picard}.

\begin{cor}\label{cor:estimate_L_by_Af}
	Let $f \in W^{1,n}_\loc(\R^n, \R^n)$ be a non-constant continuous function that satisfies \eqref{eq:QRvalue_at_infty_eucl}, where $K \geq 1$ and $\Sigma \in L^1(\R^n) \cap L^{1+\eps}_\loc(\R^n)$ for some $\eps > 0$. Let $u = v \circ f$, where $v$ is as in \eqref{eq:v_def}. Then for every open ball $B \subset \R^n$ and every $L > 0$, we have
	\[
	\int_{B \cap \{u < L\}} \abs{d u}^n \lesssim_n K^{n} L^n + K L \smallabs{A_f}(2B) + C(n)\Sigma(\R^n).
	\]
\end{cor}
\begin{proof}
	Fix $B = \B^n(x_0, r)$ with $x_0 \in \R^n$ and $r > 0$. We select a cutoff function $\eta \in C^\infty_0(\R^n, [0, 1])$ such that $\eta \equiv 1$ on $B$, $\spt \eta \subset 2B$, and $\norm{d\eta}_{L^\infty} \leq 2r^{-1}$. 
	
	Since $v$ is $C^1$, the chain rule of $C^1$ and Sobolev functions yields that
	\[
		\abs{du} = \abs{d f^* v} = \abs{f^* dv} \leq (\abs{dv} \circ f) \abs{Df}.
	\]
	Hence,
	\[
		\int_{B \cap \{u < L\}} \abs{d u}^n \leq \int_{\{u < L\}} \eta^n \bigl(\abs{dv}^n \circ f\bigr) \abs{Df}^n.
	\]
	We can then use Lemma \ref{lem:sublevel_Caccioppoli}, obtaining
	\begin{multline*}
		\int_{\{u < L\}} \eta^n \bigl(\abs{dv}^n \circ f\bigr) \abs{Df}^n\\
		\lesssim_n K^{n} L^n \int_{\{u < L\}} \abs{d \eta}^n 
		+ K L \int_{\{u < L\}} \eta^n \dd \smallabs{A_f}
		+ C(n) \int_{\{u < L\}} \eta^n \Sigma\\
		\leq K^n L^n \norm{d \eta}^n_{L^n} + K L \smallabs{A_f}(2B) + C(n) \Sigma(\R^n).
	\end{multline*}
	Since $\norm{d \eta}_{L^n} \leq C(n)$ by our assumptions that $\norm{d \eta}_{L^\infty} \leq 2r^{-1}$ and $\spt \eta \subset 2B$, the claim follows.
\end{proof}

Besides Corollary \ref{cor:estimate_L_by_Af}, we also require a couterpart to \cite[Lemma 4.2]{Bonk-PoggiCorradini_Rickman-Picard}, which in the quasiregular setting yields $A_f(B) \lesssim_{n, K} \sup_{2B} u^{n-1}$ for every ball $B$ with $u = v \circ f$. An estimate based on $\sup_{2B} u^{n-1}$ is however insufficient for us, since we do not have that every component of $u^{-1} (t, \infty)$ is unbounded for every $t > 0$. Instead, we define a pseudosupremum of a continuous function $\phi \colon \R^n \to [0, \infty)$ as follows:
\begin{multline}\label{eq:pseudosup_def}
	\pseudosup_E \phi \\ = \sup \{ t \in \R : E \text{ meets an unbounded component of } \phi^{-1}(t, \infty)\}.
\end{multline} 
This is similar to the classical $\esssup_E \phi$, which instead requires that $E$ contains a positive-measured subset of $\phi^{-1}(t, \infty)$. For bounded $E$, we clearly have $0 \leq \pseudosup_E \phi \leq \sup_E \phi < \infty$ for every continuous $\phi \colon \R^n \to [0, \infty)$. Moreover, if $E_1 \subset E_2$, then $\pseudosup_{E_1} \phi \leq \pseudosup_{E_2} \phi$. We also note that $(\pseudosup_E \phi)^p = \pseudosup_E (\phi^p)$ for $p \geq 0$, allowing us to ignore this distinction in our notation.

The pseudosupremum combines with Lemma \ref{lem:superlevel_bdd_component} to produce the following result.

\begin{lemma}\label{lem:pseudosup_superlevel_estimate}
	Let $f \in W^{1,n}_\loc(\R^n, \R^n)$ be a non-constant continuous function that satisfies \eqref{eq:QRvalue_at_infty_eucl}, where $K \geq 1$ and $\Sigma \in L^1(\R^n) \cap L^{1+\eps}_\loc(\R^n)$ for some $\eps > 0$. Let $u = v \circ f$, where $v$ is as in \eqref{eq:v_def}. Then for every open ball $B \subset \R^n$, every $\eta \in C^\infty_0(B)$ with $\eta \geq 0$, and every $L > \pseudosup_{B} u$, we have
	\[
		\int_{\{u \geq  L\}} \eta^n \bigl(\abs{dv}^n \circ f\bigr) \abs{Df}^n \lesssim_{n}  \norm{\eta}_{L^\infty} \Sigma(\R^n).
	\]
\end{lemma}
\begin{proof}
	let $U = \{u > L\}$. By definition, $B$ meets only bounded components of $U$; denote the union of these components of $U$ that meet $B$ by $U_{B}$. Now, recalling that $\spt \eta \subset B$, that $\abs{dv} \circ f = H'(\abs{f})$ by \eqref{eq:v_abs_value}, and that $f$ satisfies \eqref{eq:QRvalue_at_infty_eucl}, we use Lemma \ref{lem:superlevel_bdd_component} with $\tilde{\Sigma} = C(n)(1 + \abs{f}^2)^{n/2} \Sigma$, $\Phi(t) = [H'(t)]^n$, and $y_0 = 0$, obtaining the estimate
	\begin{multline*}
		\int_{\{u > L\}} \eta^n \bigl(\abs{dv}^n \circ f\bigr) \abs{Df}^n 
		\leq \norm{\eta}_{L^{\infty}} \int_{U_{B}} \bigl(\abs{dv}^n \circ f\bigr) \abs{Df}^n\\
		\lesssim_{n} \norm{\eta}_{L^{\infty}} \int_{U_{B}} \bigl(\abs{dv}^n \circ f\bigr) 
		\bigl(1 + \abs{f}^2\bigr)^{\frac{n}{2}} \Sigma.
	\end{multline*}
	Since also $J_f = 0$ a.e.\ in $\{u = L\}$ due to image of this set under $f$ having zero Hausdorff $n$-measure, we also have $\abs{Df}^n \lesssim_{n} (1 + \abs{f}^2)^{n/2} \Sigma$ a.e.\ in $\{u = L\}$ by \eqref{eq:QRvalue_at_infty_eucl}. Hence, we may improve the previous estimate to
	\[
		\int_{\{u \geq  L\}} \eta^n \bigl(\abs{dv}^n \circ f\bigr) \abs{Df}^n \lesssim_{n} \norm{\eta}_{L^{\infty}}  \int_{U_{2B} \cup \{u = L\}} \bigl(\abs{dv}^n \circ f\bigr) \bigl(1 + \abs{f}^2\bigr)^{\frac{n}{2}} \Sigma.
	\]
	Moreover, we again have $(\abs{dv} \circ f) (1 + \abs{f}^2)^{1/2} \leq C(n)$ by \eqref{eq:v_abs_value} and \eqref{eq:H_estimates}. Hence, we obtain the desired
	\[
		\int_{\{u \geq  L\}} \eta^n \bigl(\abs{dv}^n \circ f\bigr) \abs{Df}^n \lesssim_{n} \norm{\eta}_{L^{\infty}} \Sigma(U_{2B} \cup \{u = L\}) \leq  \norm{\eta}_{L^{\infty}} \Sigma(\R^n).
	\]
\end{proof}

With this, we prove our couterpart to \cite[Lemma 4.2]{Bonk-PoggiCorradini_Rickman-Picard}.

\begin{lemma}\label{lem:estimate_Af_by_L}
	Let $f \in W^{1,n}_\loc(\R^n, \R^n)$ be a non-constant continuous function that satisfies \eqref{eq:QRvalue_at_infty_eucl}, where $K \geq 1$ and $\Sigma \in L^1(\R^n) \cap L^{1+\eps}_\loc(\R^n)$ for some $\eps > 0$. Let $u = v \circ f$, where $v$ is as in \eqref{eq:v_def}. Then for every open ball $B \subset \R^n$, we have
	\[
		\abs{A_f}(B) \lesssim_{n} K^{n-1} \pseudosup_{2B} u^{n-1} + C(n) \bigl(\Sigma(\R^n) + [\Sigma(\R^n)]^\frac{n-1}{n}\bigr).
	\]
\end{lemma}
\begin{proof}
	We again fix an open ball $B$ with radius $r > 0$, and select a cutoff function $\eta \in C^\infty_0(\R^n, [0, 1])$ such that $\eta \equiv 1$ on $B$, $\spt \eta \subset 2B$, and $\norm{d \eta(x)}_{L^\infty} \leq 2r^{-1}$.
	
	We first estimate that
	\[
		\abs{A_f}(B) \leq \int_{\R^n} \eta^n \dd \smallabs{A_f} \leq \int_{\R^n} \eta^n \dd A_f + 2 A_f^{-}(2B)
	\]
	By \ref{eq:Afminus_estimate}, we have $A_f^{-}(2B) \lesssim_{n} \Sigma(B)$. On the other hand, we recall that by Lemma \ref{lem:v_properties}, we have $d(\abs{dv}^{n-2} \hodge dv) = s_n^* \vol_{\S^n}$. Hence, we obtain
	\begin{multline*}
		\int_{\R^n} \eta^n \dd A_f = \int_{\R^n} \eta^n f^* s_n^* \vol_{\S^n} = \int_{\R^n} \eta^n f^* d(\abs{dv}^{n-2} \hodge dv)\\
		 = \int_{\R^n} \eta^n d f^*(\abs{dv}^{n-2} \hodge dv)
		\leq n \int_{\R^n} \eta^{n-1} \abs{d \eta} \smallabs{f^* (\abs{dv}^{n-2} \hodge dv)},
	\end{multline*}
	where the commutation of $d$ and $f^*$ is valid since the form $\abs{dv}^{n-2} \hodge dv$ is $C^1$-smooth. Furthermore, we may estimate using \eqref{eq:pullback_norm_estimate} that 
	\[
		\smallabs{f^* (\abs{dv}^{n-2} \hodge dv)} \leq \abs{Df}^{n-1} (\abs{dv}^{n-1} \circ f).
	\]
	Consequently, a use of H\"older's inequality yields
	\begin{multline*}
		\int_{\R^n} \eta^{n-1} \abs{d \eta} \smallabs{f^* (\abs{dv}^{n-2} \hodge dv)}
		\\ \leq \left( \int_{\R^n} \abs{d \eta}^n \right)^{\frac{1}{n}} \left( \int_{\R^n} \eta^n \bigl(\abs{dv}^{n} \circ f\bigr) \abs{Df}^n  \right)^{\frac{n-1}{n}}.
	\end{multline*}
	Moreover, by  $\spt \eta \subset 2B$ and our estimate $\abs{d \eta}^n \leq 4^n r^{-n}$, the integral of $\abs{d \eta}^n$ is bounded from above by a constant $C(n)$. In conclusion,
	\begin{multline}\label{eq:Af_estimate_step_1}
		\abs{A_f}(B) \leq \int_{\R^n} \eta^n \dd \smallabs{A_f}\\ \lesssim_n \left( \int_{\R^n} \eta^n \bigl(\abs{dv}^{n} \circ f\bigr) \abs{Df}^n  \right)^{\frac{n-1}{n}} + C(n) \Sigma(\R^n).
	\end{multline}

	We then proceed to estimate the integral in \eqref{eq:Af_estimate_step_1}. Let $L > \pseudosup_{2B} u$. By Lemma \ref{lem:pseudosup_superlevel_estimate}, we obtain
	\begin{equation*}
		\int_{\{u \geq  L\}} \eta^n \bigl(\abs{dv}^{n} \circ f\bigr) \abs{Df}^n \lesssim_{n} \Sigma(\R^n).
	\end{equation*}
	In the remaining set $\{u < L\}$ we use Lemma \ref{lem:sublevel_Caccioppoli}, which, recalling that $\norm{d \eta}_{L^n} \leq C(n)$, yields the estimate
	\begin{multline}\label{eq:Af_estimate_step_2}
		\int_{\R^n} \eta^n \bigl(\abs{dv}^{n} \circ f\bigr) \abs{Df}^n\\
		\lesssim_n K^{n} L^n 
		+ K L \int_{\R^n} \eta^n \dd \smallabs{A_f}
		+ C(n) \Sigma(\R^n).
	\end{multline}
	
	Next, chaining together \eqref{eq:Af_estimate_step_1} and \eqref{eq:Af_estimate_step_2} and using the elementary inequality $(a+b)^p \lesssim_p a^p + b^p$ for $a, b, p \geq 0$, we obtain
	\begin{multline}\label{eq:Af_estimate_step_3}
		\int_{\R^n} \eta^n \dd \smallabs{A_f}
		\lesssim_n K^{n-1} L^{n-1} 
		+ (KL)^\frac{n-1}{n} \left( \int_{\R^n} \eta^n \dd \smallabs{A_f} \right)^\frac{n-1}{n}\\ 
		+ C(n) \left(\Sigma(\R^n) + [\Sigma(\R^n)]^\frac{n-1}{n}\right).
	\end{multline}
	We then apply Young's inequality, obtaining
	\[
		C(n) (KL)^\frac{n-1}{n} \left( \int_{\R^n} \eta^n \dd \smallabs{A_f} \right)^\frac{n-1}{n} \leq  \frac{[C(n)]^n K^{n-1} L^{n-1}}{n} + \frac{n-1}{n} \int_{\R^n} \eta^n \dd \smallabs{A_f},
	\]
	where the last integral is finite and can hence be absorbed to the left side of \eqref{eq:Af_estimate_step_3}. In conclusion, we obtain
	\[
		\abs{A_f}(B) \leq \int_{\R^n} \eta^n \dd \smallabs{A_f}
		\lesssim_n K^{n-1} L^{n-1} + C(n) \left(\Sigma(\R^n) + [\Sigma(\R^n)]^\frac{n-1}{n}\right),
	\]
	and as $L > \pseudosup_{2B} u$ was arbitrary, the claim follows.
\end{proof}

\subsection{Existence of unbounded components}

To finish this section, we point out that if $u = v \circ f$ with $f$ as in the previous section, and if we assume that $\smallabs{A_f}(\R^n) = \infty$, then every $u^{-1}(t, \infty)$ has an unbounded component. The result is a relatively immediate consequence of Lemma \ref{lem:estimate_Af_by_L}.

\begin{lemma}\label{lem:unbounded_components_infinite_Af}
	Let $f \in W^{1,n}_\loc(\R^n, \R^n)$ be a non-constant, unbounded, continuous function that satisfies \eqref{eq:QRvalue_at_infty_eucl}, where $K \geq 1$ and $\Sigma \in L^1(\R^n) \cap L^{1+\eps}_\loc(\R^n)$ for some $\eps > 0$. Let $u = v \circ f$, where $v$ is as in \eqref{eq:v_def}. Then for every $t > 0$, there exists $s = s(n, K, \Sigma(\R^n), t) > 0$ such that if $\smallabs{A_f}(B) > s$ for some ball $B \subset \R^n$, then $2B$ meets an unbounded component of $u^{-1}(t, \infty)$. In particular, if $\smallabs{A_f}(\R^n) = \infty$, then for every $t > 0$ the set $u^{-1}(t, \infty)$ has an unbounded component.
\end{lemma}
\begin{proof}
	Let $B$ be a ball, and let $s, t > 0$, with the purpose of fixing $s$ later. Suppose that $\smallabs{A_f}(B) > s$, and that $2B$ meets no unbounded component of $u^{-1}(t, \infty)$. Then $\pseudosup_{2B} u \leq t$, and Lemma \ref{lem:estimate_Af_by_L} yields
	\[
		s \leq C(n) K^{n-1} t + C(n)\bigl(\Sigma(\R^n) + [\Sigma(\R^n)]^\frac{n-1}{n}\bigr).
	\]
	Hence, we may set $s(n, K, t)$ to be bigger than the right hand side of the above estimate, and the claim follows.
\end{proof}

\section{The proof of Theorem \ref{thm:Bonk-PoggiCorradini-part}}\label{sect:main_proof}

Following the proofs of the Caccioppoli-type estimates in Section \ref{sect:caccioppoli}, we then proceed to show that the Picard theorem for quasiregular values is true when $\smallabs{A_f}(\R^n) = \infty$, assuming $\Sigma \in L^1(\R^n) \cap L^{1+\eps}_\loc(\R^n)$. For this part of the result, we're able to follow the proof of Bonk and Poggi-Corradini from \cite{Bonk-PoggiCorradini_Rickman-Picard} relatively closely, with the main difference being our use of the pseudosupremum $\pseudosup$ instead of the usual maximum.

We begin by recalling a key tool in the proof that is colloquially referred to as \emph{Rickman's hunting lemma}. For further details including the proof of the lemma, we refer to \cite[Lemma 2.1 and p.\ 627]{Bonk-PoggiCorradini_Rickman-Picard}.

\begin{lemma}[Rickman's Hunting Lemma]\label{lem:Rickman's_hunting_lemma}
	Let $\mu$ be a (non-negative) Borel measure on $\R^n$ such that $\mu(\R^n) = \infty$, $\mu(B) < \infty$ for every ball $B \subset \R^n$, and $\mu$ has no atoms. Then there exists a constant $D = D(n) > 1$ and a sequence of balls $B_j$, $j \in \Z_{>0}$ such that $\mu(8B_j) \leq D \mu(B_j)$ and $\lim_{j \to \infty} \mu(B_j) = \infty$.
\end{lemma}

We also recall a lemma on conformal capacity that is essentially similar to \cite[Lemma 5.3]{Bonk-PoggiCorradini_Rickman-Picard} but phrased in a more abstract way; this more general formulation will become relevant in the next section. Recall that if $E, F$ are compact disjoint subsets of $\R^n$, then the \emph{(conformal) capacity} of the condenser $(E, F)$ is defined by e.g.\
\begin{equation}\label{eq:capacity_def}
	\capac(E, F) = \inf \left\{ \int_{\R^n} \abs{d \eta}^n : \eta \in C^\infty_0(\R^n), \eta\vert_{E} \geq 1, \eta\vert_{F} \leq 0\right\}.
\end{equation}
By a standard convolution approximation argument, an equivalent definition is obtained if the assumption $\eta \in C^\infty_0(\R^n)$ in \eqref{eq:capacity_def} is replaced by $\eta \in W^{1,n}_0(\R^n) \cap C(\R^n)$. We call a function $\eta \in W^{1,n}_0(\R^n) \cap C(\R^n)$ with $\eta\vert_E \geq 1$ and $\eta\vert_F \leq 0$ \emph{admissible} for the condenser $(E, F)$. 

\begin{lemma}\label{lem:capacity_lower_bound_abstract}
	Let $q \geq 2$. For each $k \in \{1, \dots, q\}$, let $E_k$ and $F_k$ be closed subsets of $\R^n$ such that $E_k \cap F_k = \emptyset$ for every $k$ and $F_l \cup F_k = \R^n$ whenever $l \neq k$. Suppose that $B = \B^n(x_0, r)$ meets an unbounded component of $E_k$ for every $k \in \{1, \dots, q\}$. Let $t > 1$, and define
	\begin{equation}\label{eq:Et_Ft_defs_abstract}
		E_{k, t} = E_k \cap (\overline{tB} \setminus B), \qquad F_{k, t} = F_k \cap (\overline{tB} \setminus B).
	\end{equation}
	Then we have
	\[
		\sum_{k=1}^q \capac(E_{k,t}, F_{k,t}) \gtrsim_n q^\frac{n}{n-1}\log t.
	\]
\end{lemma}
\begin{proof}
	If $l \neq k$, we observe that since $F_l \cup F_k = \R^n$ and $E_l \cap F_l = \emptyset$, we have $E_l \subset \R^n \setminus F_l \subset F_k$. Consequently, $B$ also meets an unbounded component of $F_k$ for every $k \in \{1, \dots, q\}$, as due to our assumption that $q \geq 2$, we may select $l \in \{1, \dots, q\} \setminus \{k\}$ and note that $B$ meets an unbounded component of $E_l \subset F_k$. It follows that $(\partial sB) \cap E_k^1 \neq \emptyset \neq (\partial sB) \cap F_k$ for every $s \geq 1$, and we may thus use a capacity estimate given e.g.\ in \cite[Lemma 3.3]{Bonk-PoggiCorradini_Rickman-Picard} to conclude that
	\[
		\capac(E_{k,t}, F_{k,t}) \gtrsim_n \int_{1}^{t} \frac{r \dd s}{\left[\cH^{n-1}((\partial sB) \setminus (E_k \cup F_k))\right]^\frac{1}{n-1}}.
	\]
	We note that the denominator $\cH^{n-1}((\partial sB) \setminus (E_k \cup F_k))$ in the above integral is non-zero for every $s\geq1$; indeed, $(\partial sB) \setminus (E_k \cup F_k))$ is an open subset of $\partial sB$, and $(\partial sB) \setminus (E_k \cup F_k))$ is non-empty since $\partial sB$ is connected and $E_k$ and $F_k$ are disjoint closed sets.
	
	We then observe that the sets $\R^n \setminus (E_k \cup F_k)$ are pairwise disjoint, since $(\R^n \setminus (E_k \cup F_k)) \cap (\R^n \setminus (E_l \cup F_l)) \subset \R^n \setminus (F_k \cup F_l) = \emptyset$ whenever $k \neq l$. Thus, the sets $(\partial sB) \setminus (E_k \cup F_k)$ are disjoint for every $s \geq 1$, and Hölder's inequality for sums yields that
	\begin{multline*}
		q = \sum_{k=1}^q \left[\cH^{n-1}((\partial sB) \setminus (E_k \cup F_k))\right]^\frac{1}{n}\frac{1}{\left[\cH^{n-1}((\partial sB) \setminus (E_k \cup F_k))\right]^{\frac{1}{n}}}\\
		\leq \left[\cH^{n-1}(\partial sB)\right]^\frac{1}{n}\left( \sum_{k=1}^q \frac{1}{\left[\cH^{n-1}((\partial sB) \setminus (E_k \cup F_k))\right]^\frac{1}{n-1}} \right)^\frac{n-1}{n}.
	\end{multline*}
	Since $[\cH^{n-1}(\partial sB)]^{1/n} \lesssim_n (rs)^{(n-1)/n}$, we hence obtain the desired
	\begin{multline*}
		\sum_{k=1}^q \capac(E_{k,t}, F_{k,t}) \gtrsim_n \int_{1}^{t} \sum_{k=1}^q \frac{r ds}{\left[\cH^{n-1}((\partial sB) \setminus (E_k \cup F_k))\right]^\frac{1}{n-1}}\\
		\gtrsim_n \int_{1}^{t} \frac{q^\frac{n}{n-1} r\dd s}{rs} = q^\frac{n}{n-1} \log t.
	\end{multline*}	
\end{proof}

Now, we begin the proof of Theorem \ref{thm:Bonk-PoggiCorradini-part}. We recall the statement for the convenience of the reader.

\mainthmBCpart*

\begin{proof}
	Suppose that $h \in W^{1,n}_\loc(\R^n, \S^n)$ is continuous and has a $(K, \Sigma)$-quasiregular value with respect to the spherical metric at $q$ distinct points $w_1, \dots, w_{q} \in \partial h(\R^n)$, yet $h \notin W^{1,n}(\R^n, \S^n)$. Our objective is hence to find an upper bound on $q$ that only depends on $n$ and $K$. We may assume $q \geq 2$. Since $w_k \in \partial h(\R^n)$, by the single-value Reshetnyak's theorem for spherical quasiregular values given in Theorem \ref{thm:Picard_for_QR_values_spherical} \ref{item:Reshetnyak_sphere}, we conclude that $w_k \notin h(\R^n)$.
	
	For every point $w_k$, we select an isometric rotation $R_k \colon \S^n \to \S^n$ that takes $w_k$ to $s_n(\infty)$, and denote $h_k = R_k \circ h$. Since $R_k$ is an orientation-preserving isometry of $\S^n$, it follows that $h_k$ has a $(K, \Sigma)$-quasiregular value with respect to the spherical metric at $s_n(\infty)$. 
	
	Consequently, we obtain maps $f_k \in W^{1,n}_\loc(\R^n, \R^n)$ such that $h_k = s_n \circ f_k$. Notably, for every $k \in \{1, \dots, q\}$ and every measurable $E \subset \R^n$, we have
	\[
		A_{f_k}(E) = \int_E f_k^* s_n^* \vol_{\S^n} = \int_E h_k^* \vol_{\S^n} = \int_E h^* R_k^* \vol_{\S^n} = \int_E h^* \vol_{\S^n}.
	\]
	That is, every $A_{f_k}$ is the same measure; for convenience, we denote this measure by $A_f$. Moreover, since we assumed that $h \notin W^{1,n}(\R^n, \S^n)$, we have $\norm{Dh}_{L^n} = \infty$, and since also $\Sigma(\R^n) < \infty$, \eqref{eq:spherical_QR_value} yields that
	\[
		\abs{A_f}(\R^n) \geq A_f(\R^n) = \int_{\R^n} J_h \geq \frac{1}{K} \left( \int_{\R^n} \abs{Dh}^n - \pi^n \int_{\R^n} \Sigma \right) = \infty.
	\]
	We also note that since every $h_k$ has a $(K, \Sigma)$-quasiregular value with respect to the spherical metric at $s_n(\infty)$, we obtain that every $f_k$ satisfies \eqref{eq:QRvalue_at_infty_eucl} by Lemma \ref{lem:QRvalue_sphere_single}, allowing us to use the results of Section \ref{sect:caccioppoli} on $f_k$.
	
	We then let $u_k = v \circ f_k$ for every $k \in \{1, \dots, q\}$, where $v$ is as in \eqref{eq:v_def}. We note that the sets $s_n(\{\infty\} \cup v^{-1}(t, \infty))$ form a neighborhood basis of $s_n(\infty)$, where the neighborhoods become smaller as $t > 0$ increases. Hence, there exists $C_0 = C_0(n, w_1, w_2, \dots, w_q) > 0$ such that for every $t \geq C_0$, the sets $u_k^{-1}(t, \infty), k \in \{1, \dots, q\}$ are pairwise disjoint. Moreover, by Lemma \ref{lem:unbounded_components_infinite_Af}, there exists $A_0 = A_0(n, K, \Sigma(\R^n), w_1, w_2, \dots, w_q)$ such that if $B \subset \R^n$ is a ball with $A_f(B) > A_0$, then $2B$ meets an unbounded component of each of the sets $u_k^{-1}(3C_0, \infty)$.
	
	Since $\smallabs{A_f}(\R^n) = \infty$, we may also use Rickman's Hunting Lemma \ref{lem:Rickman's_hunting_lemma} to obtain $B_j \subset \R^n$ such that $\lim_{j \to \infty} \smallabs{A_f}(B_j) = \infty$ and $\smallabs{A_f}(8B_j) \lesssim_n \smallabs{A_f}(B_j)$. Then there exists $j_0 > 0$ such that $\smallabs{A_f}(B_j) > A_0$ whenever $j \geq j_0$. For all such $j$ and for every $k \in \{1, \dots, q\}$, we define
	\begin{equation*}
		L_{j,k} = \pseudosup_{2B_j} u_k
	\end{equation*}
	We also define
	\begin{align*}
		E_{k}^j &= u_k^{-1} [2L_{j,k}/3, \infty),& F_{k}^j &= u_k^{-1} [0, L_{j,k}/3],
	\end{align*}
	and 
	\begin{align*}
		E_{k,2}^j &= E_{k}^j \cap (\overline{4B_j} \setminus 2B_j), & F_{k,2}^j &= E_{k}^j \cap (\overline{4B_j} \setminus 2B_j).
	\end{align*}

	We claim that for every $j \geq j_0$, the sets $E_{k}^j$ and $F_{k}^j$ with $k \in \{1, \dots q\}$ satisfy the assumptions of Lemma \ref{lem:capacity_lower_bound_abstract}. Indeed, it is clear from the definition that $E_k^j \cap F_k^j = \emptyset$ for every $k$. Since $A_f(B_j) > A_0$, $2B_j$ meets an unbounded component of $u_k^{-1}(3C_0, \infty)$, and hence $L_{j,k} \geq 3C_0 > 0$ for every $k$. Thus, the sets $\R^n \setminus F_k^j = u_k^{-1}(L_{j,k}/3, \infty)$ are pairwise disjoint, and consequently $F_k^j \cup F_l^j = \R^n$ whenever $k \neq l$. Since $0 < L_{j,k} = \pseudosup_{2B_j} u_k$, we also have that $2B_j$ meets an unbounded component of every $u_k^{-1}(2L_{j,k}/3, \infty)$, and consequently $2B_j$ also meets an unbounded component of every $E_k^j$. Thus, the assumptions of Lemma \ref{lem:capacity_lower_bound_abstract} are satisfied, and it follows that for every $j \geq j_0$, we have
	\begin{equation}\label{eq:capacity_ineq_in_practice}
		\sum_{k=1}^q \capac(E_{k,2}^j, F_{k,2}^j) \gtrsim_n q^\frac{n}{n-1}.
	\end{equation}

	With \eqref{eq:capacity_ineq_in_practice} shown, we begin estimating. Let $j \geq j_0$.
	By using Lemma \ref{lem:estimate_Af_by_L} on $f_k$, we obtain
	\begin{equation}\label{eq:Af_from_above}
		\smallabs{A_f}(B_j) \lesssim_n K^{n-1} L_{j,k}^{n-1} + C(n, \Sigma(\R^n)).
	\end{equation}
	for every $k \in \{1, \dots, q\}$. Notably, since $\lim_{j \to \infty} \smallabs{A_f}(B_j) = \infty$ by our use of Rickman's Hunting Lemma, we conclude from \eqref{eq:Af_from_above} that
	\begin{equation}\label{eq:L_minimum_limit}
		\lim_{j \to \infty} \min_{k} L_{j,k} = \infty.
	\end{equation}

	We may then select a function $\psi_j \in C^\infty_0(8B_j)$ such that $\norm{\nabla \psi_j}_{L^n} \leq C(n)$ and $\psi_j \equiv 1$ on a neighborhood of $\overline{4B_j}$. Now, the function
	\[
	\eta_j = \left(\frac{3 \min(u_k, L_{j,k})}{L_{j,k}}  - 1\right)\psi
	\]
	is admissible for the condenser $(E_{k,2}^j, F_{k,2}^j)$. It follows that
	\begin{equation*}
	\capac(E_{k,2}^j, F_{k,2}^j) \leq \int_{\R^n} \abs{\nabla \eta_j}^n
	\lesssim_n \norm{\nabla \psi_j}_{L^n}^n + \int_{4B_j \cap \{u_k < L_{j,k}\}} \dfrac{\abs{\nabla u_k}^n}{\displaystyle L_{j,k}^{n}}
	\end{equation*}
	We apply Corollary \ref{cor:estimate_L_by_Af} to the last integral and use $\norm{\nabla \psi_j}_{L^n} \lesssim_n 1 \leq K^n$, obtaining
	\begin{equation*}
		\capac(E_{k,2}^1, E_{k,2}^2)
		\lesssim_n
		2K^n + \frac{K \smallabs{A_f}(8B_j)}{L_{j,k}^{n-1}} + \frac{C(n, \Sigma(\R^n))}{L_{j,k}^n}.
	\end{equation*}
	
	Due to \eqref{eq:capacity_ineq_in_practice}, there always exists a $k = k(j, h) \in \{1, \dots, q\}$ such that $\capac(E_{k,2}^1, E_{k,2}^2) \geq C(n) q^{1/(n-1)}$. Hence, for this specific choice of $k$, we have
	\begin{equation*}
		q^{\frac{1}{n-1}} \lesssim_n
		2K^n + \frac{K \smallabs{A_f}(8B_j)}{L_{j,k}^{n-1}} + \frac{C(n, \Sigma(\R^n))}{L_{j,k}^n}.
	\end{equation*}
	We then apply \eqref{eq:Af_from_above} and the estimate $\smallabs{A_f}(8B_k) \lesssim_n \smallabs{A_f}(B_k)$ we have from our use of Rickman's Hunting lemma, obtaining
	\begin{equation}\label{eq:q_from_above}
		q^{\frac{1}{n-1}} \lesssim_n 
		3K^n + \frac{K C(n, \Sigma(\R^n))}{L_{j,k}^{n-1}} + \frac{C(n, \Sigma(\R^n))}{L_{j,k}^n}
	\end{equation}
	for our specific choice of $k = k(j, h)$. But if we now let $j \to \infty$ in \eqref{eq:q_from_above}, it follows from \eqref{eq:L_minimum_limit} that the terms involving $L_{j,k}$ vanish at the limit, and we obtain the desired
	\[
		q \leq C(n) K^{n(n-1)},
	\]
	concluding the proof.
\end{proof}

\section{The proof of Theorems \ref{thm:Picard_for_QR_values} and \ref{thm:Picard_for_QR_values_spherical}}\label{sect:infinite_measure}

In order to prove Theorems \ref{thm:Picard_for_QR_values} and \ref{thm:Picard_for_QR_values_spherical}, what remains is essentially to show that under the assumptions of Theorem \ref{thm:Picard_for_QR_values}, we have $\smallabs{A_f}(\R^n) = \infty$. As stated in the introduction, this is a short step in the quasiregular version of the proof \cite[p.631]{Bonk-PoggiCorradini_Rickman-Picard}, but grows into a significantly more complex undertaking in our setting. 

\subsection{The two cases} The starting point of our argument is that if one does not have $A_f(\R^n) = \infty$, then one essentially obtains an $L^n$-integrability condition for $\nabla \log \abs{f}$. This general idea of obtaining $L^n$-regularity for $\nabla \log \abs{f}$ when the behavior of $f$ differs from that of a quasiregular map is frequent in the proofs of other results on quasiregular values \cite{Kangasniemi-Onninen_Heterogeneous, Kangasniemi-Onninen_1ptReshetnyak}. The following lemma covers the standard case that we can reduce all other cases to.

\begin{lemma}\label{lem:two_cases_lemma}
	Let $K \geq 1$ and $\Sigma \in L^{1}(\R^n) \cap L^{1+\eps}_\loc(\R^n)$ for some $\eps > 0$. Suppose that $f \in W^{1,n}_\loc(\R^n, \R^n)$ is an unbounded, continuous function such that $f$ has a $(K, \Sigma)$-quasiregular value at $0$ and $0 \notin f(\R^n)$. Then
	\[
		\abs{A_f}(\R^n) = \infty
		\qquad \text{or} \qquad
		\int_{\R^n} \frac{\abs{Df}^n}{\abs{f}^n} < \infty.
	\]
\end{lemma}
\begin{proof}
	We observe that since $f$ has a $(K, \Sigma)$-quasiregular value at $0$, we can use Lemmas \ref{lem:QRvalue_sphere} and \ref{lem:QRvalue_sphere_single} to conclude that $f$ satisfies \eqref{eq:QRvalue_at_infty_eucl}. Since $f$ is also unbounded, we may hence use the results of Section \ref{sect:caccioppoli} on $f$. We divide the proof into two main cases.
	
	\emph{Case 1:} We consider first the case where there exist $0 < s_1 < s_2 < \infty$ such that $\{\abs{f} > s_2\}$ and $\{\abs{f} < s_1\}$ both have an unbounded component. In this case, we show that $\abs{A_f}(\R^n) = \infty$. The argument is similar to the proof that $\abs{A_f}(\R^n) = \infty$ in the quasiregular case.
	
	Indeed, in this case, let $B$ be a ball that meets the unbounded components of both $\{\abs{f} > s_2\}$ and $\{\abs{f} < s_1\}$. We pick values $c_1, c_2, c_3, c_4$ such that $s_1 < c_1 < c_2 < c_3 < c_4 < s_2$. We let $E_1 = \{\abs{f} \geq c_4\}$, $F_1 = \{\abs{f} \leq c_3\}$, $E_2 = \{\abs{f} \leq c_1\}$, and $F_2 = \{\abs{f} \geq c_2\}$. Since we have $\{\abs{f} > s_2\} \subset E_1$ and $\{\abs{f} < s_1\} \subset E_2$, $B$ meets an unbounded component of $E_1$ and $E_2$. Moreover, $E_1 \cap F_1 = \emptyset = E_2 \cap F_2$ and $F_1 \cup F_2 = \R^n$. Consequently the sets $E_i$ and $F_i$ satisfy the conditions of Lemma \ref{lem:capacity_lower_bound_abstract} with $q = 2$. Hence, if $t > 1$, and $E_{i,t}$, $F_{i,t}$ are as in \eqref{eq:Et_Ft_defs_abstract}, we get
	\[
		\capac(E_{1,t}, F_{1,t}) + \capac(E_{2,t}, F_{2,t}) \gtrsim_n \log t.
	\]
	
	Consider first the case where one can find arbitrarily large values of $t$ such that $\capac(E_{1,t}, F_{1,t}) \gtrsim_n \log t$. We let $u = v \circ f$ where $v$ is as in \eqref{eq:v_def}, and select a $\psi \in C^\infty_0(2tB, [0, 1])$ with $\norm{\nabla \psi}_{L^n} \leq C(n)$ and $\psi \equiv 1$ in a neighborhood of $tB$. Similarly to the beginning of the proof of Theorem \ref{thm:Bonk-PoggiCorradini-part}, we obtain that
	\[
		\eta = \left( \frac{\min(u, H(c_4)) - H(c_3)}{H(c_4) - H(c_3)} \right) \psi
	\]
	is admissible for the condenser $(E_{1,t}, F_{1,t})$, where $H$ is as in \eqref{eq:H_def}.  We then use Corollary \ref{cor:estimate_L_by_Af} to obtain that
	\begin{multline*}
		\log t \lesssim_n \capac(E_{1,t}, F_{1,t}) \leq \int_{\R^n} \abs{\nabla \eta}^n\\ \lesssim_n \frac{1}{(H(c_4) - H(c_3))^n}  \left(H^n(c_3) \norm{\nabla \psi}_{L^n}^n + \int_{2tB \cap \{u < H(c_4)\}} \abs{\nabla u}^n \right)\\
		\leq C(n, K, c_3, c_4, \Sigma(\R^n)) + C(n, K, c_3, c_4) \smallabs{A_{f}}(4tB).
	\end{multline*}
	Letting $t \to \infty$, we conclude that $\smallabs{A_{f}}(\R^n) = \infty$. 
	
	In the other case where $\capac(E_{2,t}, F_{2,t}) \gtrsim_n \log t$ for arbitrarily large $t$, we repeat the above proof with the function
	\[
		\eta = \left( \frac{H(c_2) - \min(u, H(c_2))}{H(c_2) - H(c_1)} \right) \psi.
	\]
	Indeed, this $\eta$ is admissible for the condenser $(E_{2,t}, F_{2,t})$, and provides an analogous upper bound for $\log t$ in terms of $\abs{A_f}(4tB)$ by a similar proof.

	\emph{Case 2:} We then consider the other possible case, that there exists $s_0 \in [0, \infty]$ such that $\{\abs{f} > s\}$ has only bounded components whenever $s > s_0$, and $\{\abs{f} < s\}$ has only bounded components whenever $s < s_0$. In this case, we show that $\abs{f}^{-1} \abs{Df} \in L^n(\R^n)$. 
	
	Indeed let $s > s_0$. Since $f$ has a $(K, \Sigma)$-quasiregular value at $0$, and since $\{\abs{f} > s\}$ has only bounded components, we may use Lemma \ref{lem:superlevel_bdd_component} with $\Psi(t) = t^{-n}$ and $\tilde{\Sigma} = \abs{f}^n \Sigma$ to conclude that
	\[
		\int_{\{\abs{f} > s\}} \frac{\abs{Df}^n}{\abs{f}^n} \leq C(n) \int_{\{\abs{f} > s\}} \Sigma \leq C(n) \Sigma(\R^n).
	\]
	Monotone convergence consequently yields that
	\[
		\int_{\{\abs{f} > s_0\}} \frac{\abs{Df}^n}{\abs{f}^n} 
		\lesssim_n \Sigma(\R^n) <	\infty.
	\]
	
	We then consider the map $\tilde{f} = \iota \circ f$, where $\iota$ is the conformal inversion across the unit $(n-1)$-sphere. Then since we have $0 \notin f(\R^n)$, we obtain that $\tilde{f} \in C(\R^n, \R^n) \cap W^{1,n}_\loc(\R^n, \R^n)$ and $0 \notin \tilde{f}(\R^n)$. By using the conformality of $\iota$ and the fact that $\abs{\iota(y)} = \abs{y}^{-1}$ and $\abs{D\iota(y)} = \abs{y}^{-2}$, we obtain that
	\[
		\frac{\smallabs{D\tilde{f}}}{\smallabs{\tilde{f}}} = \frac{(\abs{D\iota} \circ f)\abs{Df}}{\abs{f}^{-1}} = \frac{\abs{Df}}{\abs{f}}.
	\]
	It also follows that the map $\tilde{f}$ also has a $(K, \Sigma)$-quasiregular value at $0$, since
	\[
		\smallabs{D\tilde{f}}^n = \frac{\abs{Df}^n}{\abs{f}^{2n}} \leq \frac{K J_f}{\abs{f}^{2n}} + \frac{\Sigma}{\abs{f}^n} = KJ_{\tilde{f}} + \smallabs{\tilde{f}}^n \Sigma.
	\] 
	Furthermore, for every $\tilde{s} > s_0^{-1}$, we have that $\{\smallabs{\tilde{f}} > \tilde{s}\} = \{\abs{f} < s^{-1}\}$ has no unbounded components. Hence, we may similarly as before use Lemma \ref{lem:superlevel_bdd_component} to obtain that
	\begin{equation*}
		\int_{\{\abs{f} < s_0\}} \frac{\abs{Df}^n}{\abs{f}^n}
		= \int_{\{\smallabs{\tilde{f}} > s_0^{-1}\}} \frac{\smallabs{D\tilde{f}}}{\smallabs{\tilde{f}}} \leq C(n)\Sigma(\R^n) < \infty.
	\end{equation*}
	In conclusion,
	\[
		\int_{\{\abs{f} \neq s_0\}} \frac{\abs{Df}^n}{\abs{f}^n} \lesssim_n \Sigma(\R^n) < \infty.
	\]
	
	It remains to show that if $0 < s_0 < \infty$, then the integral of $\abs{f}^{-n} \abs{Df}^n$ over $\{\abs{f} = s_0\}$ is finite. If $s_0 \in \{0, \infty\}$, then this set is empty. Otherwise, for a.e.\ $x \in \{\abs{f} = s_0\}$, we may estimate as follows:
	\[
		\frac{\abs{Df(x)}^n}{\abs{f(x)}^n} = \frac{\abs{Df(x)}^n}{s_0^n} \leq \frac{K J_f(x)}{s_0^n} + \Sigma(x).
	\] 
	Here, $\Sigma$ has finite integral over $\R^n$, and $J_f = 0$ a.e.\ in $\{\abs{f} = s_0\}$ due to the set having an image with zero Hausdorff $n$-measure. The proof of the lemma is hence complete.
\end{proof}

\subsection{Induced mapping and higher regularity}

The result of Lemma \ref{lem:two_cases_lemma} brings us into contact with prior ideas from \cite{Kangasniemi-Onninen_Heterogeneous}. In particular, suppose that a map $f \in C(\R^n, \R^n) \cap W^{1,n}_\loc(\R^n, \R^n)$ has a $(K, \Sigma)$-quasiregular value at $0$, with $\Sigma \in L^1(\R^n) \cap L^{1+\eps}_\loc(\R^n)$ for some $\eps > 0$, and suppose also that $0 \notin f(\R^n)$. Consider the map
\begin{equation}\label{eq:G_def}
	G \colon \R^n \to \R \times \S^{n-1}, \quad G(x) = \left(\log \abs{f(x)}, \frac{f(x)}{\abs{f(x)}} \right).
\end{equation}
Then $G$ is continuous, and if we embed $\R \times \S^{n-1}$ isometrically to $\R^{n+1}$, we see that $G \in W^{1,n}_\loc(\R^n, \R^{n+1})$. Moreover, if we equip $\R \times \S^{n-1}$ with the standard orientation, then $G$ has a valid Jacobian $J_G$ defined a.e.\ in $\R^n$ by 
\[
	J_G \vol_n = dG_{\R} \wedge  G_{\S^{n-1}}^* \vol_{\S^{n-1}} = dG_{\R} \wedge G_{\S^{n-1}}^*\hodge d(2^{-1} \abs{x}^2).
\]

By similar computations as in \cite[Lemma 7.1]{Kangasniemi-Onninen_Heterogeneous}, we obtain that
\begin{equation}\label{eq:DG_and_JG}
	\abs{DG} = \frac{\abs{Df}}{\abs{f}},  \quad J_G = \frac{J_f}{\abs{f}^n},
\end{equation}
and therefore
\begin{equation}\label{eq:Sigma_estimate_for_G}
	\abs{DG}^n \leq K J_G + \Sigma.
\end{equation}
In particular, if $\abs{Df}/\abs{f} \in L^n(\R^n)$, then \eqref{eq:DG_and_JG} immediately yields that $\abs{DG} \in L^n(\R^n)$.

Our strategy is to show that if $\Sigma \in L^{1+\eps}(\R^n) \cap L^{1-\eps}(\R^n)$, then $G$ is bounded. The first step towards this is to show that $\abs{DG}$ also has higher integrability. The argument is a standard proof based on reverse H\"older inequalities, and has already been recounted in e.g.\ \cite[Lemma 6.1]{Kangasniemi-Onninen_1ptReshetnyak} and \cite[Section 2.1]{Dolezalova-Kangasniemi-Onninen_MGFD-cont} in similar situations. Regardless, we state the result and recall the short proof, as the previous statements do not cover the case where the target of $G$ is $\R \times \S^{n-1}$.
\begin{lemma}\label{lem:global_higher_int}
	Suppose that $G \colon \R^n \to \R \times \S^{n-1}$ is continuous, that $G \in W^{1,n}_\loc(\R^n, \R \times \S^{n-1})$, and that $\abs{DG} \in L^n(\R^n)$. If $G$ satisfies \eqref{eq:Sigma_estimate_for_G} with $\Sigma \in L^1(\R^n) \cap L^{1+\eps}(\R^n)$ for some $\eps > 0$, then there exists $\eps' \in (0, \eps)$ such that
	\[
	\int_{\R^n} \abs{DG}^{(1+\eps')n} \lesssim_{n, K} \int_{\R^n} \Sigma^{1+\eps'} < \infty.
	\]
\end{lemma}
\begin{proof}
	Let $Q$ be a cube in $\R^n$ with side length $r$. We again select a cutoff function $\eta \in C^\infty(\R^n, [0, 1])$ s.t.\ $\eta\vert_Q \equiv 1$, $\spt \eta \subset 2Q$, and $\norm{\nabla \eta}_{L^\infty} \lesssim_n r^{-1}$, where we interpret $2Q$ as the cube with the same center as $Q$ but doubled side length. First, \eqref{eq:Sigma_estimate_for_G} yields
	\[
		\int_{\R^n} \eta^n \abs{DG}^n \leq K\int_{\R^n} \eta^n J_G + \int_{\R^n} \eta^n \Sigma.
	\] 
	We then use a Caccioppoli-type inequality for functions $\R^n \to \R \times M$, where $M$ is an oriented Riemannian $(n-1)$-manifold without boundary; see \cite[Lemma 2.3]{Kangasniemi-Onninen_Heterogeneous}. That is, if $G_\R$ is the $\R$-coordinate function of $G$, we obtain
	\[
		\int_{\R^n} \eta^n J_G \leq n \int_{\R^n} \eta^{n-1} \abs{DG}^{n-1} \abs{\nabla \eta} \abs{G_\R - c}
	\]
	for every $c \in \R$. By combining these estimates, using H\"older's inequality, dividing by $r^n$, and applying the assumptions on $\eta$, we obtain
	\[
		\dashint_{Q} \abs{DG}^n \lesssim_n K r^{-1} \left( \dashint_{2Q} \abs{G_{\R} - c}^{n^2} \right)^\frac{1}{n^2} \left( \dashint_{2Q} \abs{DG}^{\frac{n^2}{n + 1}} \right)^\frac{n^2 - 1}{n^2} + \dashint_{2Q} \Sigma.
	\]
	We then use the Sobolev-Poincar\'e inequality on the first integral, obtaining
	\[
		r^{-1} \left( \dashint_{2Q} \abs{G_{\R} - c}^{n^2} \right)^\frac{1}{n^2} \lesssim_n \left( \dashint_{2Q} \abs{DG_\R}^{\frac{n^2}{n+1}} \right)^\frac{n+1}{n^2} \leq \left( \dashint_{2Q} \abs{DG}^{\frac{n^2}{n+1}} \right)^\frac{n+1}{n^2}
	\]
	when $c = (G_\R)_{2Q}$. In conclusion, we obtain the reverse H\"older inequality
	\[
		\dashint_{Q} \abs{DG}^n \lesssim_n K \left( \dashint_{2Q} \abs{DG}^{\frac{n^2}{n + 1}} \right)^\frac{n + 1}{n} + \dashint_{2Q} \Sigma.
	\]
	As this holds for all cubes $Q$, we may hence use Gehring's lemma (see e.g.\ \cite[Lemma 3.2]{Iwaniec-Gehring_Lemma}), obtaining that for some $\eps' \in (0, \eps)$ we have the estimate
	\[
		\int_{\R^n} \abs{DG}^{n(1+\eps')} \lesssim_{n, K} \int_{\R^n} \Sigma^{1 + \eps'} < \infty.
	\]
\end{proof}

\subsection{Lower integrability and boundedness}

The most natural way to continue would be to obtain a lower integrability counterpart of Lemma \ref{lem:global_higher_int}, showing that $\abs{DG} \in L^{(1-\eps')n}(\R^n)$ for some $\eps' > 0$. Indeed, with both higher and lower integrability, boundedness of $G_\R$ would follow from a standard Riesz potential estimate \cite[Lemma 7.16]{Gilbarg-Trudinger_Book}. Unfortunately, we do not currently know a way to achieve this; lower integrability in our setting was discussed in \cite[Lemma 7.2]{Kangasniemi-Onninen_Heterogeneous}, but the proof of this result has a flaw. We have recovered the result \cite[Theorem 1.3]{Kangasniemi-Onninen_Heterogeneous} that this lower integrability tool was used to prove in the corrigendum \cite{Kangasniemi-Onninen_Heterogeneous_Corrigendum}, but we have no fix for the lower integrability result itself.

However, we can still achieve boundedness for $G_\R$ similarly to how we fixed the proof of \cite[Theorem 1.3]{Kangasniemi-Onninen_Heterogeneous} in \cite{Kangasniemi-Onninen_Heterogeneous_Corrigendum}. Namely, we used the same strategy as in the attempted proof of \cite[Lemma 7.2]{Kangasniemi-Onninen_Heterogeneous}, which was originally based on ideas from \cite{Faraco-Zhong_Caccioppoli}, to prove a significantly weaker logarithmic version of the original flawed lower integrability result. This version is also enough to stand in for proper lower integrability here, though the margin by which it manages this is small enough that a more refined proof of boundedness is required. Note that even though we're only concerned with the $L^p$-regularity of $\Sigma$ in this paper, we have to consider the logarithmic scale of lower integrability here, as our argument stops working on the $L^p$-scale.

In particular, the technical lower integrability result we use is as follows; we refer to \cite{Kangasniemi-Onninen_Heterogeneous_Corrigendum} for the proof.

\begin{lemma}[{\cite[Lemma 7.2 (revised)]{Kangasniemi-Onninen_Heterogeneous_Corrigendum}}]\label{lem:pseudo_lower_integrability}
	Suppose that $G \colon \R^n \to \R \times \S^{n-1}$ is continuous and non-constant, that $G \in W^{1,n}_\loc(\R^n, \R \times \S^{n-1})$, and that $\abs{DG} \in L^n(\R^n)$. If $G$ satisfies \eqref{eq:Sigma_estimate_for_G} with $\Sigma \in L^1(\R^n) \cap L^{1-\eps}(\R^n)$ for some $\eps \in (0, 1)$, then
	\[
		\int_{\R^n} \abs{DG}^n \log^n\left(1 + \frac{1}{M(\abs{DG})}\right) < \infty,
	\]
	where $M$ stands for the (centered) Hardy-Littlewood maximal function.
\end{lemma}

We then proceed to combine Lemmas \ref{lem:global_higher_int} and \ref{lem:pseudo_lower_integrability} in order to show the boundedness of the map $G$.
\begin{lemma}\label{lem:boundedness_lemma}
	Suppose that $G \colon \R^n \to \R \times \S^{n-1}$ is continuous, that $G \in W^{1,n}_\loc(\R^n, \R \times \S^{n-1})$, and that $\abs{DG} \in L^n(\R^n)$. If $G$ satisfies \eqref{eq:Sigma_estimate_for_G} with $\Sigma \in L^{1-\eps}(\R^n) \cap L^{1+\eps}(\R^n)$ for some $\eps \in (0, 1)$, then the $\R$-component $G_\R$ of $G$ is bounded.
\end{lemma}
\begin{proof}
	We may clearly assume that $G$ is non-constant, the claim is trivial for a constant function. Thus, by Lemma \ref{lem:global_higher_int} we have $\abs{DG} \in L^{n + \eps'}(\R^n)$ for some $\eps' > 0$, and by Lemma \ref{lem:pseudo_lower_integrability}, we have $\abs{DG} \log(1+M^{-1}(\abs{DG})) \in L^n(\R^n)$.
	
	We fix $x_0 \in \R^n$, with aim to estimate $\abs{G_\R(x_0) - G_\R(0)}$. We base the proof on a standard chain of balls -argument used in e.g.\  \cite{Hajlasz-Koskela-SobolevmetPoincare}. In particular, for all $i \in \Z$, we let $r_i = \abs{x_0}2^{-\abs{i} - 2}$, and select balls $B_i$, where $B_i = \B^n(2^{-\abs{i}-1} x_0, r_i)$ for $i \leq 0$ and $B_i = \B^n((1-2^{-\abs{i}-1})x_0, r_i)$ for $i \geq 0$. See Figure \ref{fig:chain_of_balls} for an illustration. 
	\begin{figure}[h]
		\centering
		\begin{tikzpicture}[scale=2.5]
			\draw (1,0) circle (1/2);
			
			\draw (1/2,0) circle (1/4);
			\draw (1/4,0) circle (1/8);
			\draw (1/8,0) circle (1/16);
			\draw (1/16,0) circle (1/32);
			\draw (1/32,0) circle (1/64);
			
			\draw (2-1/2,0) circle (1/4);
			\draw (2-1/4,0) circle (1/8);
			\draw (2-1/8,0) circle (1/16);
			\draw (2-1/16,0) circle (1/32);
			\draw (2-1/32,0) circle (1/64);
						
			\filldraw[black] (2,0) circle (0.01);
			\filldraw[black] (0,0) circle (0.01);
						
			\node at (0,0) [anchor=east] {$0$};
			\node at (2,0) [anchor=west] {$x_0$};
		\end{tikzpicture}
		\caption{The chain of balls $B_i$ from $0$ to $x_0$.}\label{fig:chain_of_balls}
	\end{figure}
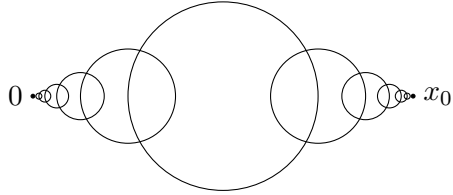

	The balls form a chain where the center of $B_{i}$ is on the boundary of $B_{i - \sgn(i)}$ for $i \neq 0$. Moreover, no point in $\R^n$ is contained in more than two balls $B_i$, and the overlap of consecutive balls $B_i \cap B_{i-\sgn(i)}$ contains a ball $B_i'$ with radius $r_i' = r_i/2$. By continuity, we also have that the integral averages $(G_\R)_{B_i}$ converge to $G_\R(0)$ as $i \to -\infty$, and to $G_\R(x_0)$ as $i \to \infty$.
	
	We thus obtain a telescopic sum estimate
	\[
		\abs{G_\R(x_0) - G_\R(0)} \leq \sum_{i=-\infty}^\infty \abs{(G_\R)_{B_{i+1}} - (G_\R)_{B_{i}}}.
	\]
	We show here the estimate for the upper end $i \geq 0$ of the series, as the estimate for the lower end $i < 0$ is analogous. By taking advantage of the ball $B_{i+1}'$ contained in $B_i \cap B_{i+1}$ and by using the Sobolev-Poincar\'e inequality, we obtain
	\begin{align*}
		\abs{(G_\R)_{B_{i+1}} - (G_\R)_{B_{i}}} 
		&\leq \big\lvert(G_\R)_{B_{i+1}'} - (G_\R)_{B_{i}}\big\rvert 
			+ \big\lvert(G_\R)_{B_{i+1}'} - (G_\R)_{B_{i+1}}\big\rvert\\
		&\leq \dashint_{B'_{i+1}} \abs{G_\R - (G_\R)_{B_{i}}}
			+ \dashint_{B'_{i+1}} \abs{G_\R - (G_\R)_{B_{i+1}}}\\
		&\leq 4^n \dashint_{B_{i}} \abs{G_\R - (G_\R)_{B_{i}}}
			+ 2^n \dashint_{B_{i+1}} \abs{G_\R - (G_\R)_{B_{i+1}}}\\
		&\lesssim_n r_i \, \dashint_{B_{i}} \abs{DG} + r_{i+1} \, \dashint_{B_{i+1}} \abs{DG}.
	\end{align*}
	Thus,
	\begin{equation*}
		\sum_{i=0}^\infty \abs{(G_\R)_{B_{i+1}} - (G_\R)_{B_{i}}} \lesssim_n \sum_{i=0}^\infty r_i^{-(n-1)} \int_{B_{i}} \abs{DG}.
	\end{equation*}

	Since $r_i$ is decreasing with respect to $i$ when $i \geq 0$ and tends to zero as $i \to \infty$, there exists an $i_0 \in \Z_{\geq 0}$ such that $r_i \leq 2$ when $i \geq i_0$, and $r_i > 2$ when $0 \leq i < i_0$. Thus, the end of the series can now be estimated using H\"older's inequality, yielding 
	\begin{multline*}
		\sum_{i=i_0}^\infty r_i^{-(n-1)} \int_{B_{i}} \abs{DG}
		\lesssim_n \sum_{i=i_0}^\infty r_i^{\frac{\eps'}{n+\eps'}} \left( \int_{B_i} \abs{DG}^{n+\eps'} \right)^\frac{1}{n+\eps'}\\
		\leq \norm{DG}_{L^{n+\eps'}} \sum_{i=i_0}^\infty r_i^{\frac{\eps'}{n+\eps'}} \lesssim_{n, \eps'} \norm{DG}_{L^{n+\eps'}} r_{i_0}^{\frac{\eps'}{n+\eps'}} \leq 2\norm{DG}_{L^{n+\eps'}}.
	\end{multline*}
	In particular, this upper bound for the end of the series is finite by our use of Lemma \ref{lem:global_higher_int}, and the upper bound is also independent on $x_0$.
	
	For the beginning part $0 \leq i < i_0$, we use the following elementary inequality: if $\Phi_1, \Phi_2$ are positive-valued real functions on an interval $I \subset \R$ with $\Phi_1$ increasing and $\Phi_2$ decreasing, then
	\[
		1 \leq \max \left( \frac{\Phi_1(a)}{\Phi_1(b)}, \frac{\Phi_2(a)}{\Phi_2(b)} \right) \leq  \frac{\Phi_1(a)}{\Phi_1(b)} + \frac{\Phi_2(a)}{\Phi_2(b)}
	\]
	for all $a, b \in I$. We use this with $I = (0, \infty)$, $\Phi_1(t) = t^{n-1}$, $\Phi_2(t) = \log(1 + t^{-1})$, $a = M(\abs{DG})(x)$, and $b = r_i^{-1/2}$ for some $0 \leq i < i_0$. We obtain
	\[
		1 \leq r_i^\frac{n-1}{2} M^{n-1}(\abs{DG}) + \frac{\log \left( 1 + 1/M(\abs{DG}) \right)}{\log \left( 1 + \sqrt{r_i} \right)}.
	\]
	Moreover, we observe that $\log (1 + \sqrt{r_i}) > \log(\sqrt{r_i}) = \log(r_i)/2$, and that due to $0 \leq i < i_0$, we have $r_i > 2$, and consequently $\log(r_i) > 0$. Hence, we conclude that whenever $0 \leq i < i_0$, we have
	\[
		1 \leq r_i^\frac{n-1}{2} M^{n-1}(\abs{DG}) + \frac{2}{\log(r_i)} \log \left( 1 + \frac{1}{M(\abs{DG})} \right),
	\]
	and in particular,
	\begin{multline}\label{eq:beginning_part_of_chain_sum}
		\sum_{i=0}^{i_0 - 1} r_i^{-(n-1)} \int_{B_{i}} \abs{DG}
		\leq \sum_{i=0}^{i_0 - 1} r_i^{-\frac{n-1}{2}} \int_{B_{i}} \abs{DG} M^{n-1}(\abs{DG})\\
		+ 2 \sum_{i=0}^{i_0 - 1} \frac{1}{r_i^{n-1}\log(r_i)} \int_{B_{i}} \abs{DG} \log \left( 1 + \frac{1}{M(\abs{DG})} \right).
	\end{multline}

	We then utilize the fact that $i_0$ is the first index for which $r_i \leq 2$, from which it follows that $r_i > 2^{i_0 - i}$ when $0 \leq i < i_0$. Thus, we may estimate the first sum on the right hand side of \eqref{eq:beginning_part_of_chain_sum} by
	\[
		 \sum_{i=0}^{i_0 - 1} r_i^{-\frac{n-1}{2}} \int_{B_{i}} \abs{DG} M^{n-1}(\abs{DG})
		 \leq \left( \int_{\R^n} M^n(\abs{DG}) \right) \sum_{j=1}^\infty 2^{-\frac{n-1}{2}j},
	\]
	which is again a finite upper bound independent on $x_0$ due to the Hardy-Littlewood maximal inequality. For the other sum on the right hand side of \eqref{eq:beginning_part_of_chain_sum}, we use both the integral and sum versions of H\"older's inequality, the fact that no point of $\R^n$ is contatined in more than two balls $B_i$, and the above estimate $r_i > 2^{i_0 - i}$, in order to obtain
	\begin{align*}
		&\sum_{i=0}^{i_0 - 1} \frac{1}{r_i^{n-1}\log(r_i)} \int_{B_{i}} \abs{DG} \log \left( 1 + \frac{1}{M(\abs{DG})} \right)\\
		&\quad \lesssim_n \sum_{i=0}^{i_0 - 1} \frac{1}{\log(r_i)} \left( \int_{B_{i}} \abs{DG}^n \log^n \left( 1 + \frac{1}{M(\abs{DG})} \right) \right)^\frac{1}{n}\\
		&\quad \leq \left( \sum_{i=0}^{i_0 - 1} \frac{1}{\log^\frac{n}{n-1}(r_i)} \right)^\frac{n-1}{n} \left( \sum_{i=0}^{i_0 - 1} \int_{B_{i}} \abs{DG}^n \log^n \left( 1 + \frac{1}{M(\abs{DG})} \right) \right)^\frac{1}{n}\\
		&\quad\leq \left( \sum_{j=1}^{\infty} \frac{1}{(\log(2) j)^\frac{n}{n-1}} \right)^\frac{n-1}{n}  \left( 2\int_{\R^n} \abs{DG}^n \log^n \left( 1 + \frac{1}{M(\abs{DG})} \right) \right)^\frac{1}{n}.
	\end{align*}
	This upper bound is also independent of $x_0$, and is finite thanks to Lemma \ref{lem:pseudo_lower_integrability}. Thus, combining our estimates, we have an $x_0$-independent upper bound for the upper end $i \geq 0$ of the telescopic sum of integral averages. An identical argument proves a similar bound for the lower end $i < 0$, completing the proof.
\end{proof}

\subsection{Completing the proofs}

It remains to complete the proofs of Theorems \ref{thm:Picard_for_QR_values} and \ref{thm:Picard_for_QR_values_spherical}. We start with Theorem \ref{thm:Picard_for_QR_values_spherical}, where we recall the statement for the convenience of the reader.

\mainthmsphere*

\begin{proof}
	Suppose that $h \in W^{1,n}_\loc(\R^n, \S^n)$ has a $(K, \Sigma)$-quasiregular value with respect to the spherical metric at $q$ distinct points $w_1, \dots, w_q \in \partial f(\R^n)$, with $q \geq 2$. By Theorem \ref{thm:Bonk-PoggiCorradini-part}, we must have either $q \leq q_0(n, K)$, or $\abs{Dh} \in L^n(\R^n)$. Suppose then that we are in the latter case, with the aim of deriving a contradiction. From this point onwards, we may ignore all of the spherical quasiregular values $w_i$ except the first two, $w_1$ and $w_2$.
	
	By the single-point Reshetnyak's theorem given in Theorem \ref{thm:Picard_for_QR_values_spherical} \ref{item:Reshetnyak_sphere}, we have $w_1, w_2 \notin h(\R^n)$. By post-composing $h$ with an isometric spherical rotation, we may assume that $w_2 = s_n(\infty)$. In this case, we have an unbounded continuous map $\tilde{f} \in W^{1,n}_\loc(\R^n, \R^n)$ such that $h = s_n \circ \tilde{f}$. We let $y_1 \in \R^n$ be the point for which $s_n(y_1) = w_1$. It follows that $\tilde{f}$ is unbounded, that $y_1 \in \partial \tilde{f}(\R^n)$, and that $\tilde{f}$ has a $(K, C(n)\Sigma)$-quasiregular value at $y_1$ by Lemma \ref{lem:QRvalue_sphere}. The fact that $\abs{Dh} \in L^n(\R^n)$ also yields that
	\[
	\smallabs{A_{\tilde{f}}} (\R^n) = \int_{\R^n} \abs{J_h} \leq \int_{\R^n} \abs{Dh}^n < \infty.
	\]
	
	We then consider the map $f = \tilde{f} - y_1$. It follows that $f$ is a continuous, unbounded map in $W^{1,n}_\loc(\R^n, \R^n)$, that $0 \notin f(\R^n)$, and that $f$ has a $(K, C(n)\Sigma)$-quasiregular value at $0$. Moreover, since $J_f = J_{\tilde{f}}$ and $1+\smallabs{f}^2 \gtrsim_{n, y_1} 1 + \smallabs{f-y_1}^2$, we obtain
	\[
	A_f(\R^n) = \int_{\R^n} \frac{2^n J_f}{(1 + \smallabs{f}^2)^n} \lesssim_{n, y_1} \int_{\R^n} \frac{2^n \smallabs{J_{f}}}{(1 + \smallabs{f - y_1}^2)^n} = \smallabs{A_{\tilde{f}}} (\R^n) < \infty.
	\]
	Thus, we may apply Lemma \ref{lem:two_cases_lemma} on $f$, and conclude that $\abs{f}^{-1} \abs{Df} \in L^n(\R^n)$.
	
	Let then $G$ be as in \eqref{eq:G_def}. Since $f$ omits 0, it follows that $G$ is a well-defined continuous map, $G \in W^{1,n}_\loc(\R^n, \R \times \S^{n-1})$, and $\abs{DG} = \abs{f}^{-1} \abs{Df} \in L^n(\R^n)$. Since also $\Sigma \in L^{1+\eps}(\R^n) \cap L^{1-\eps}(\R^n)$, it follows by Lemma \ref{lem:boundedness_lemma} that $G_\R = \log \abs{f}$ is bounded. This is a contradiction, since $f$ is unbounded. The proof is hence complete.
\end{proof}

Theorem \ref{thm:Picard_for_QR_values} is then an immediate corollary of Theorem \ref{thm:Picard_for_QR_values_spherical}. We recall the statement and give the short proof.

\mainthm*

\begin{proof}
	Suppose that $f \in W^{1,n}_\loc(\R^n, \R^n)$ is continuous and has a $(K, \Sigma)$-quasiregular value at $q$ distinct points $y_1, \dots, y_q \in \partial f(\R^n)$. Let $h = s_n \circ f$. Then by Lemma \ref{lem:QRvalue_sphere}, $h$ has a $(K, \tilde{\Sigma})$-quasiregular value with respect to the spherical metric at each of the points $s_n(y_1), \dots, s_n(y_q) \in \partial h(\R^n)$, where $\tilde{\Sigma} = C(n, y_1, \dots, y_q) \Sigma$. Now, Theorem \ref{thm:Picard_for_QR_values_spherical} yields an upper bound on $q$ dependent only on $n$ and $K$, completing the proof.
\end{proof}

\begin{rem}\label{rem:deriving_Rickman's_Picard}
	With Theorems \ref{thm:Picard_for_QR_values} and \ref{thm:Picard_for_QR_values_spherical} shown, we conclude this section by briefly pointing out how the standard Rickman's Picard Theorem follows almost immediately from the case $\Sigma \equiv 0$ of our main results. Besides Theorem \ref{thm:Picard_for_QR_values}, the only other result of quasiregular theory used in the argument is either the Liouville theorem or Reshetnyak's Theorem; the single-value versions from Theorem \ref{thm:single_value_results} can also be used for this.
	
	Both arguments begin in the same manner. Suppose towards contradiction that $f \colon \R^n \to \R^n$ is an entire non-constant $K$-quasiregular map that omits $q+1$ distinct points $y_1, \dots y_{q+1} \notin f(\R^n)$, where $q = q(n, K)$ is as in Theorem \ref{thm:Picard_for_QR_values}. We note that $f$ has a $(K, 0)$-quasiregular value at every $y \in \R^n$. Hence, by Theorem \ref{thm:Picard_for_QR_values}, we obtain that $\partial f(\R^n)$ contains at most $q$ points. Since $f$ omits more than $q$ different points, the set $\R^n \setminus \overline{f(\R^n)}$ must be non-empty. 
	
	For the argument based on Reshetnyak's theorem, we use it to conclude that the set $\intr(f(\R^n))$ is also non-empty. It follows that $\partial f(\R^n)$ separates two non-empty subsets of $\R^n$, in which case the set $\partial f(\R^n)$ has topological dimension at least $(n-1)$; see e.g.\ \cite[Theorem~IV~4]{Hurewics-Wallman_book}. This is impossible, since $\partial f(\R^n)$ has topological dimension 0 due to it containing at most $q$ points, completing the proof.
	
	For the argument based on the Liouville theorem, we instead use the non-emptiness of $\R^n \setminus \overline{f(\R^n)}$ to select a point $y_0 \in \R^n \setminus \overline{f(\R^n)}$, and post-compose $f$ with a M\"obius transformation which takes $y_0$ to $\infty$. Now, since $f$ is a quasiregular map that omits a neighborhood of $y_0$, the rotated map $\tilde{f}$ is a quasiregular map that omits a neighborhood of $\infty$, and thus $\tilde{f}$ is bounded. Hence, the Liouville theorem implies that $\tilde{f}$ is constant, resulting in a contradiction and completing the proof.
\end{rem}

\section{The planar case}\label{sect:planar}

In this section, we prove Theorem \ref{thm:q_is_2_when_n_is_2}. The result is in fact derived directly from Theorem \ref{thm:Picard_for_QR_values} with the use of a trick.

Before beginning the proof, we recall the following corollary of the single-value Reshetnyak's theorem from \cite{Kangasniemi-Onninen_1ptReshetnyak}, which generalizes the version of the argument principle used by Astala and P\"aiv\"arinta \cite[Proposition 3.3~b)]{Astala-Paivarinta}. 

\begin{lemma}[{\cite[Corollary 1.6]{Kangasniemi-Onninen_1ptReshetnyak}}]\label{lem:local_index_counting}
	Let $f_1, f_2 \in W^{1,n}_\loc(\R^n, \R^n)$ be such that both $f_i$ have a $(K_i, \Sigma_i)$-quasiregular value at $y_0 \in \R^n$, with $K_i \geq 1$ and $\Sigma_i \in L^{1+\eps}_\loc(\R^n)$ for some $\eps > 0$. Suppose that
	\[
		\liminf_{x \to \infty} \abs{f_2(x) - y_0} \neq 0 \quad \text{and} \quad \liminf_{x \to \infty} \abs{f_1(x) - f_2(x)} = 0.
	\] 
	Then
	\[
		\sum_{x \in f_1^{-1}\{y_0\}} i(x, f_1) = \sum_{x \in f_2^{-1}\{y_0\}} i(x, f_2).
	\]
\end{lemma}

We also recall a version of the main structure theorem for planar maps with a quasiregular value. A proof for the result can essentially be found embedded within \cite[Proof of Theorem 8.5.1]{Astala-Iwaniec-Martin_Book}. We regardless go over the key ideas of the argument.

\begin{lemma}\label{lem:planar_QRval_structure}
	Suppose that $f \colon \C \to \C$ has a $(K, \Sigma)$-quasiregular value at $z_0 \in \C$, where $K \geq 1$ and $\Sigma \in L^{1+\eps}(\C) \cap L^{1-\eps}(\C)$ for some $\eps > 0$. Then $f$ is of the form
	\[
	f(z) = z_0 + g(z)e^{\theta(z)},
	\]
	where $g \colon \C \to \C$ is an entire quasiregular map, and $\theta \in C(\C, \C)$ with $\lim_{z \to \infty} \theta(z) = 0$.
\end{lemma}
\begin{proof}
	We first rewrite \eqref{eq:QRvalue_def} in the form of a Beltrami equation. Indeed, recalling that $\abs{Df} = \abs{f_z} + \abs{f_{\overline{z}}}$ and $J_f = \abs{f_z}^2 - \abs{f_{\overline{z}}}^2$, we have
	\[
	\abs{f_z}^2 + \abs{f_{\overline{z}}}^2 \leq \abs{Df}^2 \leq K(\abs{f_z}^2 - \abs{f_{\overline{z}}}^2) + \abs{f - z_0}^2 \Sigma.
	\]
	Rearranging, we have
	\[
	\abs{f_{\overline{z}}}^2 \leq \frac{K-1}{K+1} \abs{f_z}^2 + \abs{f - z_0}^2 \frac{\Sigma}{K + 1}.
	\]
	Due to the elementary inequality $\sqrt{a^2 + b^2} \leq \abs{a} + \abs{b}$, we hence have
	\begin{equation}\label{eq:QRval_beltrami_ineq}
		\abs{f_{\overline{z}}} \leq k \abs{f_z} + \sigma \abs{f - z_0}, 
	\end{equation}
	where 
	\[
	k = \sqrt{\frac{K - 1}{K + 1}} \in [0, 1) \quad \text{and} \quad \sigma = \sqrt{\frac{\Sigma}{K + 1}} \in L^{2 + 2\eps}(\C) \cap L^{2 - 2\eps}(\C).
	\]
	Moreover, \eqref{eq:QRval_beltrami_ineq} can be rewritten as a Beltrami-type equation
	\begin{equation}\label{eq:QRval_beltrami_eq}
		f_{\overline{z}} = \mu f_z + A (f - z_0),
	\end{equation}
	where $\norm{\mu}_{L^\infty} \leq k < 1$ and $A \in L^{2 + 2\eps}(\C, \C) \cap L^{2 - 2\eps}(\C, \C)$. 
	
	To prove the structure theorem, one first studies the auxiliary equation
	\begin{equation}\label{eq:QRval_aux_beltrami_eq}
		\theta_{\overline{z}} = \mu \theta_z + A.
	\end{equation}
	By standard existence theory of Beltrami-type equations, one can find a solution for \eqref{eq:QRval_aux_beltrami_eq} by $\theta = \cC (I - \mu \cS)^{-1} A$, where $\cC$ is the Cauchy transform and $\cS$ is the Beurling transform. In particular, since $A \in L^{2+2\eps}(\C, \C) \cap L^{2-2\eps}(\C, \C)$, the map $\theta$ ends up being a bounded, continuous map with $\lim_{z \to \infty} \theta = 0$: see e.g.\ \cite[Theorem 4.3.11 and Section 5.4]{Astala-Iwaniec-Martin_Book}.
	
	Then, with the solution $\theta$ of \eqref{eq:QRval_aux_beltrami_eq}, one defines $g = (f - z_0) e^{-\theta}$, in which case $f = z_0 + g e^{\theta}$. By using \eqref{eq:QRval_beltrami_eq} and \eqref{eq:QRval_aux_beltrami_eq}, one computes directly that $g_{\overline{z}} = \mu g_z$. Hence, $g$ is an entire quasiregular map, completing the argument.
\end{proof}

We then prove Theorem \ref{thm:q_is_2_when_n_is_2}. We again first recall the statement for the convenience of the reader.

\mainthmplanar*

\begin{proof}
	We first reduce the case $h \colon \C \to \S^2$ to the case $f \colon \C \to \C$. Suppose that $h \in W^{1,2}_\loc(\C, \S^2)$ has a $(K, \Sigma)$-quasiregular value with respect to the spherical metric at three distinct points $w_1, w_2, w_3 \in \partial h(\C)$. By post-composing $h$ with an isometric rotation, we may assume that $w_3 = s_2(\infty)$. The single-point Rehsetnyak's theorem given in Proposition \ref{prop:Liouville_and_Reshetnyak_spherical} \ref{item:Reshetnyak_sphere} then again yields that $s_2(\infty) \notin h(\R^n)$; indeed, otherwise $h(\R^n)$ would be a neighborhood of $s_2(\infty)$ by the openness part, contradicting $s_2(\infty) = w_3 \in \partial h(\C)$. Thus, if we define $f \colon \C \to \C$ by $s_2 \circ f = h$, then by Lemma \ref{lem:QRvalue_sphere}, $f$ has a $(K, C(h)\Sigma)$-quasiregular value at two distinct points $z_1, z_2 \in \partial f(\C)$, where $s_2(z_1) = w_1$ and $s_2(z_2) = w_2$.
	
	Suppose then towards contradiction that $f \in W^{1,2}_\loc(\C, \C)$ has a $(K, \Sigma)$-quasiregular value at two distinct points $z_1, z_2 \in \partial f(\C)$. For convenience, we may assume $z_1 = 0$ and $z_2 = 1$ by replacing $f$ with the map $(f - z_1)/(z_2 - z_1)$, an operation which only introduces a multiplicative constant $C(z_1, z_2)$ to $\Sigma$. As before, by the single-point Reshetnyak's theorem, we also have that $0, 1 \notin f(\C)$.
	
	Since $\Sigma \in L^{1+\eps}(\C) \cap L^{1-\eps}(\C)$, we may use Lemma \ref{lem:planar_QRval_structure} to write $f(z) = g(z) e^{\theta(z)}$, where $g \colon \C \to \C$ is an entire quasiregular map and $\theta \in C(\C, \C)$ with $\lim_{z \to \infty} \theta(z) = 0$. Since $f(z) \neq 0$ and $e^{\theta(z)} \neq 0$ for all $z \in \C$, we conclude that $g$ omits $0$. Hence, we may lift $g$ in the exponential map to find an entire quasiregular map $\gamma \colon \C \to \C$ such that $g = e^{\gamma}$. In particular,
	\[
	f(z) = e^{\gamma(z) + \theta(z)}.
	\]
	
	We first observe that $\gamma$ is non-constant. Indeed, suppose towards contradiction that $\gamma \equiv c$. Then we have $\lim_{z \to \infty} f(z) = e^{c}$. However, this is impossible, since it follows from $\lim_{z \to \infty} \theta(z) = 0$ that $(\partial f(\C)) \setminus f(\C) \subset \{e^{c}\}$, yet $(\partial f(\C)) \setminus f(\C)$ must at least contain the two distinct points $0$ and $1$. Hence, we conclude that $\gamma$ is non-constant; in particular, by the Picard theorem for entire quasiregular maps, $\gamma$ omits at most a single point in $\C$.
	
	\enlargethispage{2\baselineskip}
	Next, we claim that $\gamma + \theta$ has a $(K, 4\Sigma)$-quasiregular value at each of the points $2\pi ik, k \in \Z$. Indeed, we have 
	\begin{multline*}
		\abs{D(\gamma + \theta)}^2 = \frac{\abs{Df}^2}{\abs{f}^2} \leq K \frac{J_f}{\abs{f}^2} + \frac{\min(\abs{f}^2, \abs{f-1}^2)}{\abs{f}^2} \Sigma\\ 
		= K J_{\gamma + \theta} + \min\bigl(1, \big\lvert1 - e^{-\gamma -\theta}\big\rvert^2\bigr) \Sigma.
	\end{multline*}
	Now, fix $k \in \Z$, and suppose first that $\abs{\gamma(z) + \theta(z) - 2\pi ik} \leq 2^{-1}$. Then
	\begin{multline*}
		\big\lvert 1 - e^{-\gamma(z) -\theta(z)} \big\rvert = \big\lvert e^{2\pi ik -\gamma(z) -\theta(z)} - 1 \big\rvert\\ 
		\leq \abs{2\pi ik - \gamma(z) - \theta(z)} \left( \sum_{j=1}^\infty \frac{\abs{k2\pi i - \gamma(z) - \theta(z)}^{j-1}}{j!} \right)\\
		\leq \abs{2\pi ik - \gamma(z) - \theta(z)} \left( \sum_{j=1}^\infty \frac{1}{2^{j-1}j!} \right) \leq 2 \abs{2\pi ik - \gamma(z) - \theta(z)}.
	\end{multline*}
	If on the other hand we have $\abs{\gamma(z) + \theta(z) - 2\pi ik} \geq 2^{-1}$, then it follows that $1 \leq 4 \abs{\gamma(z) + \theta(z) - 2\pi ik}^2$. It follows that $\min(1, \lvert1 - e^{-\gamma -\theta}\rvert^2) \leq 4 \abs{\gamma + \theta - 2\pi ik}^2$, and in particular,
	\[
	\abs{D(\gamma + \theta)}^2 \leq K J_{\gamma + \theta} + \abs{(\gamma + \theta) - 2\pi ik}^2 4\Sigma.
	\]
	
	Now, since $\gamma + \theta$ has a $(K, 4\Sigma)$-quasiregular value at $2\pi ik$ for every $k \in \Z$, and since $\Sigma \in L^{1+\eps}(\C) \cap L^{1-\eps}(\C)$, Theorem \ref{thm:Picard_for_QR_values} provides a $q = q(n, K)$ such that $2\pi ik \in \partial [(\gamma + \theta)(\C)]$ for at most $q$ different values of $k$. Since $\gamma$ also only omits at most one point of $\C$, we can select a $k_0 \in \Z$ such that $2\pi i k_0 \in \gamma(\C)$ and $2\pi i k_0 \notin \partial [(\gamma + \theta)(\C)]$. Since $1 \notin f(\C)$, we also have $2\pi i k_0 \notin (\gamma + \theta)(\C)$, and hence there exists a radius $r_0 > 0$ such that $(\gamma + \theta)(\C) \cap \B^2(k_0 2 \pi i, r_0) = \emptyset$.
	
	Now, for the final step of the argument, we apply Lemma \ref{lem:local_index_counting}. Indeed, we have 
	\[
	\liminf_{z \to \infty} \abs{(\gamma + \theta)(z) - 2 \pi i k_0} \geq r_0 > 0
	\quad \text{and} \quad
	\lim_{z \to \infty} \abs{(\gamma + \theta)(z) - \gamma(z)} = 0.
	\] 
	Moreover, $\gamma+\theta$ has a $(K, 4\Sigma)$-quasiregular value at $2 \pi i k_0$, and $\gamma$ is a non-constant quasiregular map. Hence, we conclude that
	\[
	0 = \sum_{z \in (\gamma+\theta)^{-1}\{2\pi ik_0\}} i(z, \gamma+\theta)
	= \sum_{z \in \gamma^{-1}\{2\pi ik_0\}} i(z, \gamma) > 0,
	\]
	which is a contradiction. The proof is thus complete.
\end{proof}

\section{Counterexamples}\label{sect:counterexamples}

In this chapter, we investigate the sharpness of the assumptions of Theorem \ref{thm:Picard_for_QR_values}. In particular, we show that the assumption we made in Theorem \ref{thm:Bonk-PoggiCorradini-part} that $\Sigma \in L^1(\R^n) \cup L^{1+\eps}_\loc(\R^n)$ is not sufficient to obtain the conclusions of Theorem \ref{thm:Picard_for_QR_values}.

\begin{ex}\label{ex:plus_epsilon_counterexample}
	In our first example, we construct for every $q \in \Z_{> 0}$ a continuous map $f \in W^{1,\infty}_\loc(\R^n, \R^n)$ such that $f$ has $q$ distinct $(1, \Sigma)$-quasiregular values, where $\Sigma \in L^1(\R^n) \cap L^{1-\eps}(\R^n) \cap L^\infty_\loc(\R^n)$ for every $\eps \in (0, 1)$. See Figure \ref{fig:sea_urchin_counterexample} for a rough illustration of the example in the case $n = 2$.
	
	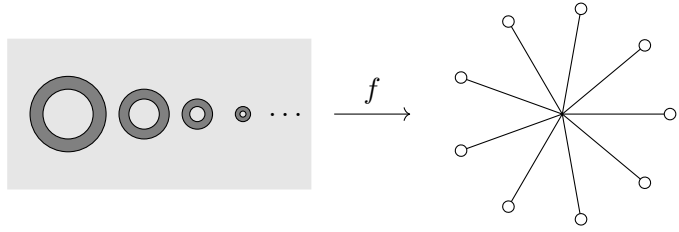
\begin{figure}[h]
		\begin{tikzpicture}
			\fill[fill=gray!20] (-0.8, 1) -- (-0.8, -1) -- (3.2, -1) -- (3.2, 1) -- (-1, 1);
			\filldraw[color=black, fill=gray](0,0) circle (0.5);
			\filldraw[color=black, fill=gray!20](0,0) circle (0.33);
			\filldraw[color=black, fill=gray](1,0) circle (0.33);
			\filldraw[color=black, fill=gray!20](1,0) circle (0.2);
			\filldraw[color=black, fill=gray](1.7,0) circle (0.2);
			\filldraw[color=black, fill=gray!20](1.7,0) circle (0.1);
			\filldraw[color=black, fill=gray](2.3,0) circle (0.1);
			\filldraw[color=black, fill=gray!20](2.3,0) circle (0.04);
			\draw (2.5,0) node[anchor=west] {$\ldots$};
			
			\draw[->] (3.5,0) -- (4, 0) node[anchor=south] {$f$}
			-- (4.5,0);
			
			\foreach \i in {1,...,9} {
				\draw[-o] (6.5,0) -- ({6.5+1.5*cos(\i*40)}, {1.5*sin(\i*40)});
			}
		\end{tikzpicture}
		\caption{\small Rough illustration of the map $f$ of Example \ref{ex:plus_epsilon_counterexample} in the case $n = 2$. The map $f$ takes each of the infinitely many shaded annuli on the domain side to one of the open-ended stalks on the target side, stopping partway through.  In the lighter shaded part of $\R^2$ the map $f$ is locally constant, with the unbounded component mapped to the center of the stalks. The tips of the stalks are quasiregular values of $f$ and are contained in $\partial f(\R^2)$.}
		\label{fig:sea_urchin_counterexample}
	\end{figure}
	
	We begin by selecting $q$ distinct points points $\{y_1, \dots, y_{q}\} \in \S^{n-1} \subset \R^n$. We let $d_0 > 0$ be the minimum distance from a point $y_k$ to a line $\{t y_l, t \in \R\}$, where $k \neq l$.
	
	We then consider the function $\theta \colon (0, 2^{-1}) \to [0, \infty)$ given by
	\[
	\theta(r) = \log^{\frac{n - 1 - \delta}{n}} \frac{1}{r},
	\]
	where $\delta \in (0, n-1)$. Note that $\theta$ is decreasing. We also define a function $\Theta \colon \B^n(0, 2^{-1}) \setminus \{0\} \to [0, \infty)$ by
	\[
	\Theta(x) = \theta(\abs{x}).
	\]
	Then we have
	\[
	\int_{\B^n(0, 2^{-1})} \abs{\nabla \Theta}^n \lesssim_{n, \delta} \int_{\B^n(0, 2^{-1})} \frac{1}{\abs{x}^n \log^{1 + \delta} \abs{x}^{-1}} < \infty.
	\]
	Thus, $\nabla \Theta \in L^n(\B^n(0, 2^{-1}))$, and consequently by H\"older's inequality, $\nabla \Theta \in L^{(1-\eps)n}(\B^n(0, 2^{-1}))$ for every $\eps \in (0, 1)$. However, we regardless have 
	\[
	\lim_{x \to 0} \Theta(x) = \lim_{r \to 0} \theta(r) = \infty.
	\]
	Thus, we may select radii $2^{-1} = R_1 > R_2 > \dots$ such that we have $\theta(R_{i+1}) - \theta(R_{i}) = i$ for all $i \in \Z_{\geq 0}$.
	
	We then pick a discrete set of points $\{x_i : i \in \Z_{> 0}\} \subset \R^n$ such that the closures of the balls $B_i = \B^n(x_i, R_{i})$ are pairwise disjoint. We also denote $B'_i = \B^n(x_i, R_{i+1})$, and $k_i = ({i \mod q}) \in \{1, \dots, q\}$. We then define a function $f \colon \R^n \to \R^n$ as follows: in $\R^n \setminus \bigcup_{i} B_i$ we have $f \equiv 0$, in $B_i \setminus B'_i$ we have
	\[
	f(x) = (1 - e^{\Theta(x-x_i) - \theta(R_{i})}) y_{k_i},
	\]
	and finally in $B_i'$ we have $f(x) \equiv (1 - e^{-i}) y_{k_i}$.
	
	By our construction, we observe that $f \in W^{1, \infty}_\loc(\R^n, \R^n)$, $f$ is continuous, and $y_j \in \partial f(\R^n)$ for every $j \in \{1, \dots, q\}$. We also have $J_f \equiv 0$ everywhere since the image of $f$ is a 1-dimensional tree, and $\abs{Df} \equiv 0$ in $\R^n \setminus \bigcup_{i} B_i$ and in every $B_i'$. Hence, we may select $\Sigma \equiv 0$ in these sets, and have $\abs{Df} \leq J_f + \abs{f - y_j}^n \Sigma$ for every $j \in \{1, \dots, q\}$.
	
	It remains to consider the regions $B_i \setminus B'_i$. In these regions, we have
	\[
	\frac{\abs{Df}}{\abs{f - y_{k_i}}} = \frac{\abs{y_{k_i}} e^{\Theta(x-x_i) - \theta(R_{i})} \abs{\nabla \Theta(x-x_i)}}{\abs{y_{k_i}} e^{\Theta(x-x_i) - \theta(R_{i})}} = \abs{\nabla \Theta(x-x_i)}.
	\]
	Moreover, whenever $j \neq k_i$, we may use $e^{\Theta(x-x_i) - \theta(R_{i})} \leq 1$, $\abs{y_j} = 1$, and $\abs{f - y_j} \geq d_0$ to obtain
	\[
	\frac{\abs{Df}}{\abs{f - y_{j}}} = \frac{\abs{y_{j}} e^{\Theta(x-x_i) - \theta(R_{i})} \abs{\nabla \Theta(x-x_i)}}{\abs{f - y_{j}}} \leq d_0^{-1} \abs{\nabla \Theta(x-x_i)}.
	\]
	Thus, we may select $\Sigma = \max(1, d_0^{-n}) \abs{\nabla \Theta(x-x_i)}^n$. Now, since the regions $B_i \setminus B_i'$ are translates of the concentric annuli $\B^n(0, R_i) \setminus \B^n(0, R_{i+1})$ by $x_i$, and since $\abs{\nabla \Theta} \in L^p(\B^n(0, R_1))$ for all $p \in (0, n]$, we obtain that $\Sigma \in L^1(\R^n) \cap L^{1-\eps}(\R^n)$ for every $\eps \in (0, 1)$. Moreover, since $\{x_i\}$ is discrete and since $\Sigma$ is bounded on every $B_i \setminus B_i'$, we get that $\Sigma \in L^\infty_\loc(\R^n)$.
\end{ex}

\begin{ex}\label{ex:minus_epsilon_counterexample}
	We then provide the complementary example, which shows the necessity of the global lower integrability assumption in Theorem \ref{thm:Picard_for_QR_values}. In particular, this time we construct for every $q \in \Z_{> 0}$ a continuous map $f \in W^{1,\infty}_\loc(\R^n, \R^n)$ with $q$ distinct $(1, \Sigma)$-quasiregular values, where $\Sigma \in L^1(\R^n) \cap L^{\infty}(\R^n)$. Our strategy is similar to the one used in Example \ref{ex:plus_epsilon_counterexample}, but we use increasingly large annuli instead of increasingly small ones.
	
	We let $\{y_1, \dots, y_{q}\} \in \S^{n-1}$ and $d_0 > 0$ be as in the previous example. This time, we consider the map $\theta \colon (2, \infty) \to [0, \infty)$ given by
	\[
	\theta(r) = \log^{\frac{n - 1 - \delta}{n}} r,
	\]
	where $\delta \in (0, n-1)$. We define $\Theta \colon \R^n \setminus \B^n(0, 2) \to [0, \infty)$ by $\Theta(x) = \theta(\abs{x})$. Similarly to last time, we have
	\[
	\int_{\R^n \setminus \B^n(0, 2)} \abs{\nabla \Theta}^n \lesssim_{n, \delta} \int_{\R^n \setminus \B^n(0, 2)} \frac{1}{\abs{x}^n \log^{1 + \delta} \abs{x}} < \infty.
	\]
	Moreover, we have $\lim_{r \to \infty} \theta(r) = \infty$ and $\abs{\nabla \Theta} \in L^\infty(\R^n \setminus \B^n(0, 2))$.
	
	We again split $\R^n \setminus \B^n(0, 2)$ into sub-annuli by picking $2 = R_1 < R_2 < \dots$ such that $\theta(R_{i+1}) - \theta(R_i) = i$. We pick $\{x_i\}$ such that the closures of the balls $B_i = \B^n(x_i, R_{i+1})$ are pairwise disjoint; note that this time $\{x_i\}$ is automatically discrete and in fact extremely sparse, as we have $\abs{x_i - x_j} \geq R_i + R_j \geq 4$ whenever $i \neq j$. We also again denote $B_i' = \B^n(x_i, R_{i})$ and $k_i = ({i \mod q}) \in \{1, \dots, q\}$.
	
	We then define $f \colon \R^n \to \R^n$ so that in the set $\R^n \setminus \bigcup_{i} B_i$ we have $f \equiv 0$, in the sets $B_i \setminus B_i'$ we have
	\[
	f(x) = (1 - e^{\Theta(x-x_i) - \theta(R_{i+1})}) y_{k_i},
	\]
	and in the sets $B_i'$ we have $f(x) = (1 - e^{-i}) y_{k_i}$. We again get that $f$ is continuous, that $y_j \in \partial f(\R^n)$ for all $j \in \{1, \dots, q\}$, that $J_f \equiv 0$, and moreover that $f \in W^{1, \infty}(\R^n, \R^n)$. In order for all $y_j$ to be $(\Sigma, 1)$-quasiregular values of $f$, we can again pick $\Sigma \equiv 0$ in $\R^n \setminus \bigcup_i B_i$ and in the sets $B_i'$. Moreover, in the sets $B_i \setminus B_i'$, a similar argument as in the last example shows that we may pick $\Sigma = \max(1, d_0^{-n}) \abs{\nabla \Theta(x-x_i)}^n$, in which case $\Sigma \in L^1(\R^n) \cap L^\infty(\R^n)$.
\end{ex}

%%%%%%%%%%%%%%%%%%%%%%%%%%%%%%%%%%%%%%% Bibliography %%%%%%%%%%%%%%%%%%%%%%%%%%%%%%%%%%%%%%%

\bibliographystyle{abbrv}
\bibliography{sources}

\end{document}